\documentclass[10pt,reqno]{amsart}
\usepackage[english]{babel}
\usepackage{latexsym}
\usepackage{amssymb,amsmath,dsfont}
\usepackage{mathrsfs}
\usepackage{amsthm}
\usepackage{graphics,graphicx,color,subfigure,float,color}
\usepackage{verbatim}
\usepackage{url}
\usepackage{array}
\usepackage[table]{xcolor}
\usepackage{amsthm}
\usepackage{a4}
\usepackage{fullpage}

\newtheorem{corollary}{Corollary}[section]

\newtheorem{lemma}[corollary]{Lemma}

\newtheorem{remark}[corollary]{Remark}
\newtheorem{theorem}[corollary]{Theorem}

\newtheorem*{ack}{Acknowledgments}
\numberwithin{equation}{section}
\newcommand{\N}{\mathbb{N}}
\newcommand{\R}{\mathbb{R}}
\begin{document}

\thispagestyle{empty}
\title[\tiny Interior feedback stabilization of wave equations with dynamic boundary delay]{Interior feedback stabilization of wave equations with dynamic boundary delay}
\author{Ka\"{\i}s Ammari}
\address{UR Analyse et Contr\^{o}le des Edp, UR 13ES64, D\'{e}partement de Math\'{e}matiques,
Facult\'{e} des Sciences de Monastir, Université de Monastir, 5019 Monastir, Tunisie}
\email{kais.ammari@fsm.rnu.tn}
\author{St\'{e}phane Gerbi}
\address{Laboratoire de Math\'{e}matiques UMR 5127 CNRS \& Universit\'{e} Savoie Mont-Blanc,
Campus scientifique,
73376 Le Bourget du Lac Cedex, France}
\email{stephane.gerbi@univ-smb.fr}
\date{} 
\thanks{
The authors wish to thank R\'{e}gion Rh\^one-Alpes for the financial support COOPERA, CMIRA 2013.
}
\begin{abstract}
In this paper we consider an interior stabilization problem for the wave equation with dynamic boundary delay.
We prove some stability results under the choice of damping operator. The proof of the main result is
based on a frequency domain method and combines a contradiction argument with the multiplier
technique to carry out a special analysis for the resolvent.
\end{abstract}

\maketitle
\tableofcontents

\vfill\break
\section{Introduction}

We study the interior stabilization of a wave equation in an open
bounded domain $\Omega$ of $\mathbb{R}^n, \, n \geq 2$. We denote by
$\partial\Omega$ the boundary of $\Omega$ and we assume that
$\partial \Omega = {\Gamma}_0 \cup {\Gamma}_1,$
where $\Gamma_0,\ \Gamma_1$ are closed subsets of $\partial \Omega$
with $\Gamma_0\cap
\Gamma_1=\emptyset$. Moreover we assume  $meas(\Gamma_0)>0.$ The
system is given by:
\begin{equation}
\left\{
\begin{array}{ll}
u_{tt}-\Delta u +a \, u_{t}=0, & x\in \Omega ,\ t>0 \,,\,\\[0.1cm]
u(x,t)=0, & x\in \Gamma _{0},\ t>0 \,,\,\\[0.1cm]
u_{tt}(x,t)=- \displaystyle\frac{\partial u}{\partial \nu }(x,t) - \mu u_{t}(x,t-\tau ) & x\in \Gamma_{1},\ t>0 \,,\,\\[0.1cm]
u(x,0)=u_{0}(x) & x\in \Omega \,,\,\\[0.1cm]
u_{t}(x,0)=u_{1}(x) & x\in \Omega \,,\, \\[0.1cm]
u_{t}(x,t-\tau)=f_{0}(x,t-\tau ) & x\in\Gamma_{1},\, t\in (0,\tau)\,,\,
\end{array}
\right.  \label{ondes}
\end{equation}
where $\nu$ stands for the unit normal vector of $\partial \Omega$ pointing towards the exterior of $\Omega$ and $\frac{\partial }{\partial\nu}$ is the normal
derivative.
Moreover, the constant $\tau >0$ is the time delay,
$a$ and $\mu$ are positive
numbers and the initial data $u_{0}\,,\,u_{1},~f_{0}$ are given functions belonging to suitable
spaces that will be precised later.

Let us first review some results for particular cases which seem to us
interesting.

In the absence of the delay term (i.e. $\tau =0$) problem (\ref{ondes})
becomes
\begin{equation}
\left\{
\begin{array}{ll}
u_{tt}-\Delta u +a \, u_{t}=0, & x\in \Omega ,\ t>0 \,,\,\\[0.1cm]
u(x,t)=0, & x\in \Gamma _{0},\ t>0 \,,\,\\[0.1cm]
u_{tt}(x,t)=- \displaystyle\frac{\partial u}{\partial \nu }(x,t) - \mu u_{t}(x,t) & x\in \Gamma_{1},\ t>0 \,,\,\\[0.1cm]
u(x,0)=u_{0}(x) & x\in \Omega \,,\,\\[0.1cm]
u_{t}(x,0)=u_{1}(x) & x\in \Omega \,,\, \\[0.1cm]
u_{t}(x,t-\tau)=f_{0}(x,t-\tau ) & x\in\Gamma_{1},\, t\in (0,\tau)\,,\,
\end{array}%
\right.  \label{onde_2}
\end{equation}

This type of problems arise (for example) in modelling of longitudinal
vibrations in a homogeneous bar in which there are viscous effects. The term
$a u_t$, indicates that the stress is proportional not only to the
strain, but also to the displacement rate (see \cite{CS76} for instance). From the
mathematical point of view, these problems do not neglect acceleration terms
on the boundary. Such type of boundary conditions are usually called \textit{%
dynamic boundary conditions}. They are not only important from the
theoretical point of view but also arise in several physical applications.
For instance in one space dimension, problem (\ref{onde_2}) can modelize the
dynamic evolution of a viscoelastic rod that is fixed at one end and has a
tip mass attached to its free end. The dynamic boundary conditions
represents the Newton's law for the attached mass (see \cite{BST64,AKS96,
CM98} for more details). In the two dimension space, as showed in \cite{G06}
and in the references therein, these boundary conditions arise when we
consider the transverse motion of a flexible membrane $\Omega $ whose
boundary may be affected by the vibrations only in a region. Also some
dynamic boundary conditions as in problem (\ref{onde_2}) appear when we
assume that $\Omega $ is an exterior domain of $\mathbb{R}^{3} $ in which
homogeneous fluid is at rest except for sound waves. Each point of the
boundary is subjected to small normal displacements into the obstacle (see
\cite{B76} for more details). This type of dynamic boundary conditions are
known as acoustic boundary conditions.

Well-posedness and longtime behavior for analogous equations as (\ref{ondes}%
) (without delay) on bounded domains have been investigated by many authors
in recent years (see, e.g.,  \cite{G94_1}, \cite{G94_2}, \cite{P08}, \cite{PS04}).

Among the early results dealing with this type of boundary conditions are
those of Grobbelaar-Van Dalsen \cite{G94_1,G94_2,G96} in which the author has made
contributions to this field.

In \cite{G94_1} the author introduced a model which describes the damped
longitudinal vibrations of a homogeneous flexible horizontal rod of length $%
L $ when the end $x = 0$ is rigidly fixed while the other end $x=L$ is free
to move with an attached load. This yields to a system of two second order
equations of the form
\begin{equation}  \label{Vad_1}
\left\{
\begin{array}{ll}
u_{tt}-u_{xx}-u_{txx}=0, & x\in (0,L),\,t>0,\\[0.1cm]
u(0,t)=u_{t}(0,t)=0, & t>0,\\[0.1cm]
u_{tt}(L,t)=-\left[ u_{x}+u_{tx}\right] (L,t), & t>0,\\[0.1cm]
u\left( x,0\right) =u_{0}\left( x\right) ,u_{t}\left( x,0\right)
=v_{0}\left( x\right) & x\in (0,L),\\
u\left( L,0\right) =\eta ,\qquad u_{t}\left( L,0\right) =\mu \ .&
\end{array}
\right.
\end{equation}
By rewriting problem (\ref{Vad_1}) within the framework of the abstract
theories of the so-called $B$-evolution theory, an existence of a unique
solution in the strong sense has been shown. An exponential decay result was
also proved in \cite{G96} for a problem related to (\ref{Vad_1}), which
describe the weakly damped vibrations of an extensible beam. See \cite{G96}
for more details.

Subsequently, Zang and Hu \cite{ZH07}, considered the problem
\begin{equation}  \label{Zang_Hu}
\left\{
\begin{array}{ll}
u_{tt}-p\left( u_{x}\right) _{xt}-q\left( u_{x}\right) _{x}=0, & x\in \left(0,1\right) ,\,t>0, \vspace{0.2cm} \\
u\left( 0,t\right) =0, & t\geq 0 \,,\,\vspace{0.2cm} \\
\left( p\left( u_{x}\right) _{t}+q\left( u_{x}\right) \left( 1,t\right)
+ku_{tt}\left( 1,t\right) \right) =0, & t\geq 0 \,,\,\vspace{0.2cm} \\
u\left( x,0\right) =u_{0}\left( x\right) ,\qquad u_{t}\left( x,0\right) =u_{1}\left( x\right) , & x\in \left( 0,1\right).
\end{array}%
\right.
\end{equation}
By using the Nakao inequality, and under appropriate conditions on $p$ and $%
q $, they established both exponential and polynomial decay rates for the
energy depending on the form of the terms $p$ and $q$.

Similarly, and
always in the absence of the delay term, Pellicer and Sol{à}-Morales \cite%
{PS04} considered the one dimensional problem of (\ref{ondes}) as an
alternative model for the classical spring-mass damper system, and by using
the dominant eigenvalues method, they proved that for small values of the
parameter $a$ the partial differential equations in problem (\ref{onde_2})
has the classical second order differential equation
\begin{equation*}
m_{1}u^{\prime \prime }(t)+d_{1}u^{\prime }(t)+k_{1}u(t)=0,
\end{equation*}%
as a limit, where the parameter $m_{1}\,,\,d_{1}\mbox{ and }k_{1}$ are
determined from the values of the spring-mass damper system. Thus, the
asymptotic stability of the model has been determined as a consequence of
this limit. But they did not obtain any rate of convergence. This result was
followed by recent works \cite{P08,PS08}. In particular in \cite{PS08}, the
authors considered a one dimensional nonlocal nonlinear strongly damped wave
equation with dynamical boundary conditions. In other words, they looked to
the following problem:
\begin{equation}
\left\{
\begin{array}{ll}
u_{tt}-u_{xx}-\alpha u_{txx}+\varepsilon f\left( u(1,t),\frac{u_{t}(1,t)}{%
\sqrt{\varepsilon }}\right) =0,\  &  \\[0.1cm]
u(0,t)=0, &  \\[0.1cm]
u_{tt}(1,t)=-\varepsilon \left[ u_{x}+\alpha u_{tx}+ru_{t}\right]
(1,t)-\varepsilon f\left( u(1,t),\frac{u_{t}(1,t)}{\sqrt{\varepsilon }}%
\right) , &
\end{array}%
\right.  \label{spring-mass}
\end{equation}
with $x\in (0,1),\,t>0,\,r,\alpha >0$ and $\varepsilon \geq 0$. The above
system models a spring-mass-damper system, where the term $\varepsilon
f\left( u(1,t),\frac{u_{t}(1,t)}{\sqrt{\varepsilon }}\right) $ represents a
control acceleration at $x=1$. By using the invariant manifold theory, the
authors proved that for small values of the parameter $\varepsilon $, the
solutions of (\ref{spring-mass}) are attracted to a two dimensional
invariant manifold. See \cite{PS08}, for further details.

The main difficulty of the problem considered is related to the non ordinary
boundary conditions defined on $\Gamma _{1}$. Very little attention has been
paid to this type of boundary conditions. We mention only a few particular
results in the one dimensional space \cite{GV99,PS04,DL02,K92}.

A related problem to (\ref{onde_2}) is the following:
\begin{eqnarray*}
u_{tt}-\Delta u+g(u_{t}) &=&f\hspace*{1.5cm}\text{ in }\Omega \times (0,T) \\
\frac{\partial u}{\partial \nu }+K(u)u_{tt}+h(u_{t}) &=&0\hspace*{1.5cm}%
\text{ on }\partial \Omega \times (0,T) \\
u(x,0) &=&u_{0}(x)\hspace*{1cm}\text{ in }\Omega \\
u_{t}(x,0) &=&u_{1}(x)\hspace*{1cm}\text{ in }\Omega
\end{eqnarray*}%
where the boundary term $h(u_{t})=|u_{t}|^{\rho }u_{t}$ arises when one
studies flows of gas in a channel with porous walls. The term $u_{tt}$ on
the boundary appears from the internal forces, and the nonlinearity $%
K(u)u_{tt}$ on the boundary represents the internal forces when the density
of the medium depends on the displacement. This problem has been studied in
\cite{DL02,DLS98}: by using the Fadeo-Galerkin approximations and a
compactness argument the authors proved the global existence and the
exponential decay of the solution of the problem.

We recall some results related to the interaction of an elastic medium with rigid mass. By using the classical semigroup theory, Littman and Markus
\cite{LM88} established a uniqueness result for a particular Euler-Bernoulli beam rigid body structure. They also proved the asymptotic stability of the
structure by using the feedback boundary damping. In \cite{LL98} the authors considered the Euler-Bernoulli beam equation which describes the dynamics of
clamped elastic beam in which one segment of the beam is made with viscoelastic material and the other of elastic material. By combining the
frequency domain method with the multiplier technique, they proved the exponential decay for the transversal motion but not for the longitudinal
motion of the model, when the Kelvin-Voigt damping is distributed only on a subinterval of the domain. In relation with this point, see also the work by
Chen et \textit{al.} \cite{CLL98} concerning the Euler-Bernoulli beam equation with the global or local Kelvin-Voigt damping. Also models of
vibrating strings with local viscoelasticity and Boltzmann damping, instead of the Kelvin-Voigt one, were considered in \cite{LL02} and an exponential
energy decay rate was established. Recently, Grobbelaar-Van Dalsen \cite{G03} considered an extensible thermo-elastic beam which is hanged at one end with
rigid body attached to its free end, i.e. one dimensional hybrid thermoelastic structure, and showed that the method used in \cite{O97} is
still valid to establish an uniform stabilization of the system. Concerning the controllability of the hybrid system we refer to the work by Castro and
Zuazua \cite{CZ98}, in which they considered flexible beams connected by point mass and the model takes account of the rotational inertia.

The purpose of this paper is to study problem (\ref{ondes}), in which a
delay term acted in the dynamic boundary conditions. In recent years one
very active area of mathematical control theory has been the investigation
of the delay effect in the stabilization of hyperbolic systems and many
authors have shown that delays can destabilize a system that is
asymptotically stable in the absence of delays (see \cite{amman,ammaniar,DLP,nicaisepignotti1,nicaisepignotti,nicaisevalein} for more
details).

As it has been proved by Datko \cite[Example 3.5]{Dat91}, systems of the
form
\begin{equation}  \label{stron_datko}
\left\{
\begin{array}{ll}
w_{tt}-w_{xx}-aw_{xxt}=0, & x\in(0,1),\,t>0, \vspace{0.3cm} \\
w\left( 0,t\right) =0,\qquad w_{x}\left( 1,t\right) =-kw_{t}\left( 1,t-\tau
\right), & t>0,%
\end{array}%
\right.
\end{equation}
where $a,\,k$ and $\tau$ are positive constants become unstable for an
arbitrarily small values of $\tau$ and any values of $a$ and $k$. In (\ref%
{stron_datko}) and even in the presence of the strong damping $-aw_{xxt}$,
without any other damping, the overall structure can be unstable. This was
one of the main motivations for considering problem (\ref{ondes})( of
course the structure of problem (\ref{ondes}) and (\ref{stron_datko}) are
different due to the nature of the boundary conditions in each problem).

Subsequently, Datko et \emph{al} \cite{DLP} treated the following one
dimensional problem:
\begin{equation}
\left\{
\begin{array}{ll}
u_{tt}(x,t)-u_{xx}(x,t)+2au_{t}(x,t)+a^{2}u(x,t)=0, & 0<x<1,\,\ t>0,\vspace{%
0.3cm} \\
u(0,t)=0, & t>0,\vspace{0.3cm} \\
u_{x}(1,t)=-ku_t(1,t-\tau ), & t>0,%
\end{array}%
\right.  \label{Dakto_system}
\end{equation}%
which models the vibrations of a string clamped at one end and free at the
other end, where $u(x,t)$ is the displacement of the string. Also, the
string is controlled by a boundary control force (with a delay) at the free
end. They showed that, if the positive constants $a$ and $k$ satisfy
\begin{equation*}
k\frac{e^{2a}+1}{e^{2a}-1}<1,
\end{equation*}%
then the delayed feedback system (\ref{Dakto_system}) is stable for all
sufficiently small delays. On the other hand if
\begin{equation*}
k\frac{e^{2a}+1}{e^{2a}-1}>1,
\end{equation*}
then there exists a dense open set $D$ in $(0,\infty)$ such that for each $%
\tau\in D$, system (\ref{Dakto_system}) admits exponentially unstable
solutions.

It is well known that in the absence of delay in (\ref{Dakto_system}), that
is for $\tau =0$, system (\ref{Dakto_system}) is uniformly asymptotically
stable under the condition $a^{2}+k^{2}>0$ and the total energy of the
solution satisfies for all $t>0,$
\begin{equation}
E(t,u)=\int_{0}^{1}(u_{t}^{2}+u_{x}^{2}+a^{2}u^{2})dx\leq CE\left(
0,u\right) e^{-\alpha t}  \label{energy_decay}
\end{equation}%
for some positive constant $\alpha.$ See \cite{C79} for more details.

Recently, Ammari et \emph{al} \cite{ampignic} have treated the $N-$dimensional
wave equation%
\begin{equation}
\left\{
\begin{array}{ll}
u_{tt}(x,t)-\Delta u(x,t)+au_{t}(x,t-\tau )=0, & x\in \Omega ,\,t>0,\vspace{%
0.3cm} \\
u(x,t)=0, & x\in \Gamma _{0},\,t>0,\vspace{0.3cm} \\
\dfrac{\partial u}{\partial \nu }(x,t)=-ku(x,t), & x\in \Gamma _{1},\,t>0,%
\vspace{0.3cm} \\
u(x,0)=u_{0}(x) & x\in \Omega \,,\,\\[0.1cm]
u_{t}(x,0)=u_{1}(x) & x\in \Omega \,,\, \\[0.1cm]
u_{t}(x,t-\tau)=f_{0}(x,t-\tau ) & x\in\Gamma_{1},\, t\in (0,\tau)\,,\,
\end{array}%
\right.  \label{Ammari_al_problem}
\end{equation}%
where $\Omega $ is an open bounded domain of $\mathbb{R}^{N},\,N\geq 2$ with
boundary $\partial \Omega =\Gamma _{0}\cup \Gamma _{1}$ and $\Gamma _{0}\cap
\Gamma _{1}=\emptyset $. Under the usual geometric condition on the domain $%
\Omega $, they showed an exponential stability result, provided that the
delay coefficient $a$ is sufficiently small.

In \cite{nicaisepignotti1} the authors examined a system of wave equation with a linear
boundary damping term with a delay. Namely, they looked to the following
system
\begin{equation}
\left\{
\begin{array}{ll}
u_{tt}-\Delta u=0, & x\in \Omega ,\ t>0 \,,\,\\[0.1cm]
u(x,t)=0, & x\in \Gamma _{0},\ t>0 \,,\,\\[0.1cm]
\displaystyle \frac{\partial u}{\partial \nu }(x,t)=\mu_{1}u_{t}(x,t)+\mu_{2}u_{t}(x,t-\tau ) & x\in \Gamma _{1},\ t>0 \,,\,\\[0.1cm]
u(x,0)=u_{0}(x), & x\in \Omega \,,\, \\[0.1cm]
u_{t}(x,0)=u_{1}(x) & x\in \Omega\,,\, \\[0.1cm]
u_{t}(x,t-\tau)=g_{0}(x,t-\tau) & x\in \Omega, \tau\in(0,1)\,,\,%
\end{array}
\right.  \label{delay_1}
\end{equation}
and proved under the assumption
\begin{equation}
\mu _{2}<\mu _{1}  \label{coeff}
\end{equation}%
(which means that the weight of the feedback with delay is smaller than the
one without delay) that null stationary state is exponentially stable. On
the contrary, if (\ref{coeff}) does not hold, they found a sequence of
delays for which the corresponding solution of (\ref{delay_1}) will be
unstable. The main approach used in \cite{nicaisepignotti1}, is an observability
inequality obtained with a Carleman estimate. 

The case of time-varying delay (i.e. $\tau=\tau(t)$ is a function depending
on $t$) has been studied by Nicaise, Valein and Fridman \cite{NVF09} in one
space dimension. In their work, an exponential stability result was given
under the condition:
\begin{equation}  \label{1_d}
\mu_{2}<\sqrt{1-d} \; \mu_{1},
\end{equation}%
where $d$ is a constant such that
\begin{equation}  \label{condition_d}
\tau^{\prime }(t)\leq d<1,\,\forall t>0.
\end{equation}

Delay effects arise in many applications and practical problems and
it is well-known that an arbitrarily small  delay may destabilize a
system which is uniformly asymptotically stable in absence of delay
(see e.g. \cite{Lag1,Lag2,Datko,DLP,Z90}).

The stability of
\eqref{ondes} with $\tau=0, a=0$ has been studied in \cite{amdenis}
where it has been shown that the system is stable under some geometric condition on $\Gamma_1$ (as in \cite{blr}).
Moreover, if $\mu=0,$  that is in absence of delay, the above problem for any $a>0$ is exponentially stable even. On the contrary, in presence of a delay term there are instability phenomena probably, as in \cite{nicaisepignotti1}.
 
Let us also cite the recent work of Ammari and Nicaise, \cite{AN-2015}, in which the authors performed a complete study of the stabilisation of elastic systems by
collocated feedback with or without delay.
 
In this paper the idea is to contrast the effect of the time delay by using the dissipative feedback (i.e., by giving the control in the
feedback form $a \, u_t(x,t)$ or $-a \Delta u_t(x,t), \ \ \ x\in \Omega,\ \ t> 0$). 

In the next section, we will show the global existence of problem \eqref{ondes} by transforming the delay term and by using a semigroup approach.
The natural question is then the stability of problem \eqref{ondes}. This is the goal of section 3. We will show that a ``shifted''  problem is asymptotically stable with a polynomial
decay rate and we cannot answer the question of the stability of problem \eqref{ondes}. In fact, in the last section, numerical experiments in 1D shows that
under certain conditions, problem \eqref{ondes} is unstable.

To stabilize problem \eqref{ondes}, we will see that a Kelvin-Voigt damping is efficient. This is done in section 4.

Lastly, we will conduct some numerical examples in 1D to illustrate these stability or instability results.
\section{Well-posedness of problem (\protect\ref{ondes})}
In this section we will first transform the delay
boundary conditions by adding a new unknown. Then we will use a semigroup approach and the Lumer-Phillips' theorem to prove the existence and uniqueness
of the solution of the problem (\ref{ondes}).

We point out that the well-posedness in evolution equations with delay is not always obtained. Recently, Dreher, Quintilla and
Racke have shown some ill-posedness results for a wide range of evolution equations with a delay term \cite{DQR09}.  

\subsection{Setup and notations}

We present here some material that we shall use in order to prove the local
existence of the solution of problem (\ref{ondes}). We denote
\begin{equation*}
H_{\Gamma_{0}}^{1}(\Omega) =\left\{u \in H^1(\Omega) /\ u_{\Gamma_{0}} =
0\right\} .
\end{equation*}
 
By $( .,.) $ we denote the scalar product in $L^{2}( \Omega)$ i.e. $(u,v)(t)
= \displaystyle \int_{\Omega} u(x,t) v(x,t) dx$. Also we mean by $\Vert
.\Vert_{q}$ the $L^{q}(\Omega) $ norm for $1 \leq q \leq \infty$, and by $%
\Vert .\Vert_{q,\Gamma_{1}}$ the $L^{q}(\Gamma_{1}) $ norm.

Let $T>0$ be a real number and $X$ a Banach space endowed with norm $\Vert
.\Vert _{X}$.

$L^{p}(0,T;X) ,\ 1 \leq p < \infty$ denotes the space of
functions $f$ which are $L^{p}$ over $\left( 0,T\right) $ with values in $X$%
, which are measurable and $\Vert f \Vert_{X} \in L^{p} \left(0,T\right)$.
This space is a Banach space endowed with the norm
\begin{equation*}
\Vert f\Vert_{L^{p}\left( 0,T;X\right) }= \left(\int_{0}^{T}\Vert
f\Vert_{X}^{p} dt\right)^{1/p} .
\end{equation*}
$L^{\infty}\left( 0,T;X\right) $ denotes the space of functions $f:\left]0,T%
\right[ \rightarrow X$ which are measurable and $\Vert f\Vert_{X}\in
L^{\infty }\left( 0,T\right) $. This space is a Banach space endowed with
the norm:
\begin{equation*}
\Vert f\Vert_{L^{\infty}(0,T;X)}=\mbox{ess}\sup_{0<t<T}\Vert f\Vert_{X} \; .
\end{equation*}
We recall that if $X$ and $Y$ are two Banach spaces such that $%
X\hookrightarrow Y$ (continuous embedding), then
\begin{equation*}
L^{p}\left( 0,T;X\right) \hookrightarrow L^{p}\left( 0,T;Y\right) , \ 1 \leq
p\leq \infty \;.
\end{equation*}


\subsection{Semigroup formulation of the problem}

In this section, we will prove the global existence and the uniqueness of
the solution of problem (\ref{ondes}). We will first transform the problem (%
\ref{ondes}) to the problem (\ref{wave2}) by making the change of variables (%
\ref{change-variable}), and then we use the semigroup approach to prove the
existence of the unique solution of problem (\ref{wave2}).

To overcome the problem of the boundary delay, we introduce the new variable:
\begin{equation}
z\left( x,\rho ,t\right) =u_{t}\left( x,t-\tau \rho \right) ,\ x\in \Gamma
_{1},\ \rho \in \left( 0,1\right) ,\ t>0.  \label{change-variable}
\end{equation}%
Then, we have
\begin{equation}
\tau z_{t}\left( x,\rho ,t\right) +z_{\rho }\left( x,\rho ,t\right) =0,~%
\text{in }\Gamma _{1}\times \left( 0,1\right) \times \left( 0,+\infty
\right) .  \label{equation-z}
\end{equation}%
Therefore, problem (\ref{ondes}) is equivalent to:
\begin{equation}
\left\{
\begin{array}{ll}
u_{tt}-\Delta u + a \, u_{t} =0, & x\in \Omega ,\ t>0 \,,\,\\[0.1cm]
u(x,t)=0, & x\in \Gamma _{0},\ t>0 \,,\,\\[0.1cm]
u_{tt}(x,t)=-  \displaystyle\frac{\partial u}{\partial \nu }(x,t) - \mu z(x,1,t), & x\in \Gamma_{1} ,\ t>0 \,,\,\\[0.1cm]
\tau z_{t}(x,\rho ,t)+z_{\rho }(x,\rho ,t)=0, & x\in \Gamma _{1},\rho \in(0,1)\,,\,t>0 \,,\\[0.1cm]
z(x,0,t)=u_{t}(x,t) & x\in \Gamma _{1},\ t>0 \,,\\[0.1cm]
u(x,0)=u_{0}(x) & x\in \Omega \,,\\
u_{t}(x,0)=u_{1}(x) & x\in \Omega \,,\\
z(x,\rho ,0)=f_{0}(x,-\tau \rho ) & x\in \Gamma _{1},\,\rho \in (0,1) \ .
\end{array}%
\right. 
 \label{wave2}
\end{equation}
 
The first natural question is the existence of solutions of the problem (\ref{wave2}).
In this section we will give a sufficient condition that guarantees that this problem is well-posed.

For this purpose we will use a semigroup formulation of the
initial-boundary value problem (\ref{wave2}). 
\textcolor{black}{
If we denote $V:=\left( u,u_{t},\gamma _{1}(u_{t}),z\right)^{T}$, we define the energy space:
\begin{equation*}
\mathscr{H}=H_{\Gamma _{0}}^{1}(\Omega )\times L^{2}\left( \Omega \right) \times L^{2}(\Gamma _{1})\times L^{2}(\Gamma_{1} \times (0,1)).
\end{equation*}%
Clearly, $\mathscr{H}$ is a Hilbert space with respect to the inner product%
\begin{equation}\label{energynorm}
\left\langle V_{1},V_{2}\right\rangle _{\mathscr{H}}=\int_{\Omega }\nabla
u_{1}.\nabla u_{2}dx+\int_{\Omega }v_{1}v_{2}dx+\int_{\Gamma _{1}}w_{1}w_{2} d\sigma +\xi \int_{\Gamma _{1}}\int_{0}^{1}z_{1}z_{2}d\rho d\sigma
\end{equation}%
for $V_{1}=(u_{1},v_{1},w_{1},z_{1})^{T}$, $V_{2}=(u_{2},v_{2},w_{2},z_{2})^{T}$ and $\xi >0$ a nonnegative real number defined later.\\
Therefore, if $V_{0} \in \mathscr{H} \mbox{ and } V \in \mathscr{H}$, the problem (\ref{wave2}) is formally 
equivalent to the following abstract evolution equation in the Hilbert space $\mathscr{H}$:
}
\begin{equation}
\left\{
\begin{array}{ll}
V'(t)=  \mathscr{A}V(t) , & t>0, \vspace{0.1cm}\\
V\left( 0\right) =V_{0}, &
\end{array}
\right.  \label{Matrix_problem}
\end{equation}
where $'$ denotes the derivative with respect to time $t$, $V_{0}:=\left( u_{0},u_{1},\gamma_{1}(u_{1}),f_{0}(.,-.\tau)\right)^{T}$
and the operator $\mathscr{A}$ is defined by:
\begin{equation*}
\mathscr{A}\left(
\begin{array}{c}
u\vspace{0.2cm} \\
v\vspace{0.2cm} \\
w\vspace{0.2cm} \\
z
\end{array}
\right) =\left(
\begin{array}{c}
\displaystyle v\vspace{0.2cm} \\
\displaystyle \Delta u- a \, v\vspace{0.2cm} \\
\displaystyle -\frac{\partial u}{\partial \nu }-\mu z\left( .,1\right) \vspace{0.2cm}\\
\displaystyle -\frac{1}{\tau }z_{\rho }
\end{array}%
\right).
\end{equation*}%

The domain of $\mathscr{A}$ is the set of $V=(u,v,w,z)^{T}$ such that:
\begin{eqnarray}
(u,v,w,z)^{T}\in \left(H_{\Gamma_{0}}^{1} (\Omega)\cap H^{2}(\Omega)\right)\times H^1_{\Gamma_0}(\Omega)\times L^{2}(\Gamma_{1})\times L^{2}\left(\Gamma_{1};H^{1}(0,1)\right) ,&& \label{domainA1}\vspace*{0.3cm} \\\label{domainA2}
w=\gamma_{1}(v) = \,z(.,0) \text{ on }\Gamma_{1}.&& \label{domainA3}
\end{eqnarray}
Let us finally define $\xi^{\star} = \displaystyle \mu \tau$. For all $\displaystyle \xi > \xi^{\star}, \mbox{ we also define } \mu_{1} = \frac{\xi}{2\tau} + \frac{\mu}{2}
\mbox{ and }\mathscr{A}_{d} = \mathscr{A} - \mu_{1} \, I$.

The well-posedness of problem (\ref{wave2}) is ensured by:
\begin{theorem}\label{existence_u} Let $V_{0}\in \mathscr{H}$, then there
exists a unique solution $V\in C\left( \mathbb{R}_{+};\mathscr{H}%
\right) $ of problem (\ref{Matrix_problem}). Moreover, if $V_{0}\in %
\mathscr{D}\left( \mathscr{A}\right) $, then
\begin{equation*}
V\in C\left( \mathbb{R}_{+};\mathscr{D}\left( \mathscr{A}\right) \right)
\cap C^{1}\left( \mathbb{R}_{+};\mathscr{H}\right) .
\end{equation*}
\end{theorem}
\begin{proof}
To prove Theorem \ref{existence_u}, we first prove that 
there exists a unique solution $V\in C\left( \mathbb{R}_{+};\mathscr{H}\right) $ of  the shifted problem:
\begin{equation}
\left\{
\begin{array}{ll}
V'(t)=  \mathscr{A}_{d}V(t) , & t>0, \vspace{0.1cm}\\
V\left( 0\right) =V_{0}, &
\end{array}
\right.  \label{Matrix_problem_Ad}
\end{equation}
Then as $\mathscr{A} = \mathscr{A}_{d} + \mu_{1} \, I$, there will exist $V\in C\left( \mathbb{R}_{+};\mathscr{H}\right) $ solution of problem (\ref{Matrix_problem}).

In order to prove the existence and uniqueness of the solution of problem \eqref{Matrix_problem_Ad} we use the
semigroup approach and the Lumer-Phillips' theorem.  

Indeed, let $V=(u,v,w,z)^{T}\in \mathscr{D}\left(\mathscr{A}\right) $. By definition of the operator $\mathscr{A}$ and
the scalar product of $\mathscr{H}$, we have:
\begin{equation*}
\begin{array}{ll}
\displaystyle \left\langle \mathscr{A}V,V\right\rangle_{\mathscr{H}} =& \displaystyle \int_{\Omega}\nabla u.\nabla v dx + \int_{\Omega }v \Delta u dx  -  \int_{\Omega }a |v(x)|^{2} dx  \\
\displaystyle &+ \displaystyle \int_{\Gamma_{1}}w\left( -\frac{\partial u}{\partial \nu }-\mu z\left(\sigma,1\right)
\right) d\sigma -\dfrac{\xi }{\tau }\int_{\Gamma _{1}}\int_{0}^{1}zz_{\rho}d\rho d\sigma . \
\end{array}
\end{equation*}
By Green's formula we  obtain: 
\begin{equation}\label{dissipative1}
\left\langle \mathscr{A}V,V\right\rangle_{\mathscr{H}} = -  \int_{\Omega }a |v(x)|^{2} dx -\mu\int_{\Gamma_{1}}z\left(\sigma,1\right) wd\sigma -\frac{\xi }{\tau }\int_{\Gamma _{1}}\int_{0}^{1}z_{\rho}zd\rho dx. 
\end{equation}
But we have:
\begin{eqnarray}
\frac{\xi }{\tau }\int_{\Gamma _{1}}\int_{0}^{1}z_{\rho }z(\sigma,\rho) \,
d\rho \, d\sigma &=&\frac{\xi }{2\tau}\int_{\Gamma _{1}} \int_{0}^{1}\frac{%
\partial }{\partial \rho }z^{2}(\sigma,\rho) \, d\rho \, d\sigma  \notag
\\
&=&\frac{\xi }{2\tau}\int_{\Gamma_{1}}\left(
z^{2}(\sigma,1,t)-z^{2}(\sigma,0) \right) d\sigma \;.
\label{Energy_second_equation}
\end{eqnarray}
Thus from the compatibility condition (\ref{domainA3}), we get: 
\begin{eqnarray*}
- \, \frac{\xi }{\tau }\int_{\Gamma _{1}}\int_{0}^{1}z_{\rho }z \,d\rho \, d\sigma &=&\frac{\xi }{2\tau}\int_{\Gamma_{1}}\left(
v^{2} -z^{2}(\sigma,1,t) \right) d\sigma \ .
\end{eqnarray*}
Therefore equation \eqref{dissipative1} becomes:
\begin{equation}\label{dissipative2}
\begin{array}{ll}
\displaystyle \left\langle \mathscr{A}V,V\right\rangle_{\mathscr{H}} &=\displaystyle -  \int_{\Omega }a |v(x)|^{2} dx 
-\frac{\xi }{2 \tau }\int_{\Gamma _{1}}\int_{0}^{1}z^{2}(\sigma,1,t) d\sigma
\displaystyle +\frac{\xi }{2 \tau }\int_{\Gamma _{1}} |v|^{2}(\sigma) d\sigma \\
~&\displaystyle -\mu\int_{\Gamma_{1}}z\left(\sigma,1\right) wd\sigma   \ .
\end{array}
\end{equation}
To treat the last term in the preceding equation,Young's inequality gives: 
\begin{equation}\label{youngkv}
-\displaystyle \int_{\Gamma_{1}}v(\sigma) z\left(\sigma,1\right) d\sigma
\leq \displaystyle  \frac{1}{2} \int_{\Gamma_{1}}z^{2}\left(\sigma,1\right) d\sigma+  \displaystyle  \frac{1}{2} \int_{\Gamma_{1}} v^{2}(\sigma,t) d\sigma \ .
\end{equation}
Therefore, we firstly get:
\begin{equation}\label{dissipative3}
\displaystyle \left\langle \mathscr{A}V,V\right\rangle_{\mathscr{H}} +\displaystyle \int_{\Omega }a |v(x)|^{2} dx 
-\left(\frac{\xi }{2 \tau} +\frac{\mu}{2}\right)\int_{\Gamma _{1}}|v(\sigma)|^{2} d\sigma
 + \displaystyle \left(\frac{\xi }{2 \tau} - \frac{\mu}{2}\right)\int_{\Gamma _{1}}z^{2}(\sigma,1) d\sigma  \leq 0 \ ,
\end{equation}
which writes
\begin{equation*}
\displaystyle \left\langle \mathscr{A}V,V\right\rangle_{\mathscr{H}} +\displaystyle \int_{\Omega }a |v(x)|^{2} dx 
-\mu_{1}\int_{\Gamma _{1}}|v(\sigma)|^{2} d\sigma
 + \displaystyle \left(\frac{\xi }{2 \tau} - \frac{\mu}{2}\right)\int_{\Gamma _{1}}z^{2}(\sigma,1) d\sigma  \leq 0 \ .
\end{equation*}

From the preceding inequality, we get:
\begin{equation}\label{dissipativeAd}
\displaystyle \left\langle \Big(\mathscr{A} - \mu_{1} \, I\Big)V,V\right\rangle_{\mathscr{H}} \leq -\displaystyle \int_{\Omega }a |v(x)|^{2} dx 
- \displaystyle \left(\frac{\xi }{2 \tau} - \frac{\mu}{2}\right)\int_{\Gamma _{1}}z^{2}(\sigma,1) d\sigma \ .
\end{equation}
As $\forall \xi > \xi^{\star} \,,\,  \displaystyle \left(\frac{\xi }{2 \tau} - \frac{\mu}{2}\right) >0$, we finally get:
\begin{equation}\label{dissipative4}
\displaystyle \left\langle \Big(\mathscr{A} - \mu_{1} \, I\Big)V,V\right\rangle_{\mathscr{H}} \leq 0 \ .
\end{equation}
Thus the operator $\mathscr{A}_{d} = \mathscr{A} - \mu_{1} \, I $ is dissipative.

Now we  want to show that  $\forall \lambda >0 \,,\, \forall \xi > \xi^{\star}\,,\, \lambda I -\mathscr{A}_{d}$ is surjective. To prove that, it is clear that it suffices to show that
$\lambda I -\mathscr{A}$ is surjective for all $\lambda > 0$.

For $ F=(f_{1},f_{2},f_{3},f_{4})^{T}\in \mathscr{H}$, let $V=(u,v,w,z)^{T}\in \mathscr{D}\left( \mathscr{A}\right) $
solution of
\begin{equation*}
\left( \lambda I-\mathscr{A}\right) V=F,
\end{equation*}%
which is:
\begin{eqnarray}
\lambda u - v &=& f_{1},\vspace{0.2cm}  \label{eqsurj1}\\
\lambda v -  \Delta u + a v  &=& f_{2}, \label{eqsurj2} \\
\lambda w + \frac{\partial u}{\partial \nu } + \mu z(.,1) &=& f_{3}, \label{eqsurj3} \\
\lambda z + \frac{1}{\tau} z_{\rho}  &=& f_{4}.\label{eqsurj4} \
\end{eqnarray}%

To find $V=(u,v,w,z)^{T}\in \mathscr{D}\left( \mathscr{A}\right) $ solution of the system (\ref{eqsurj1}), (\ref{eqsurj2}), (\ref{eqsurj3}) and (\ref{eqsurj4}), we suppose $u$ is determined with the appropriate regularity. Then from  (\ref{eqsurj1}), we get:
\begin{equation}
v=\lambda u-f_{1} \ .  \label{solution_v}
\end{equation}
Therefore, from the compatibility condition on $\Gamma_{1}$, (\ref{domainA3}), we determine $z(.,0)$ by:
\begin{equation}
z(x,0) = v( x) =\lambda u(x) -f_{1}(x) , \ \mbox{\ for } \, x \in \Gamma_{1}.
\label{z_solution}
\end{equation}
Thus, from  (\ref{eqsurj4}), $z$ is the solution of the linear Cauchy problem:
\begin{equation}\label{diffz}
\left\{
\begin{array}{ll}
z_{\rho} = \tau\Big(f_{4}(x) - \lambda z(x,\rho) \Big), & \mbox{ for } x \in \Gamma_{1} \,,\, \rho \in (0,1), \\
z(x,0)  =\lambda u(x) -f_{1}(x).&
\end{array}
\right.
\end{equation}
The solution of the Cauchy problem (\ref{diffz}) is given by:
\begin{equation}\label{z_formula}
z(x,\rho) = \lambda u( x) e^{-\lambda \rho \tau} -f_{1}e^{-\lambda \rho \tau }+\tau e^{-\lambda \rho \tau}\int_{0}^{\rho}f_{4}(x,\sigma) e^{\lambda \sigma \tau }d\sigma
\quad \mbox{ for } x \in \Gamma_{1} \,,\, \rho \in (0,1) .
\end{equation}
So, we have at the point $\rho = 1$,
\begin{equation} \label{z1}
z(x,1) = \lambda u(x)  e^{-\lambda \tau} + z_{1}(x), \quad \mbox{ for } x \in \Gamma_{1}
\end{equation}
with
$$
z_{1}(x) = -f_{1}e^{-\lambda \tau }+\tau e^{-\lambda  \tau}\int_{0}^{1}f_{4}(x,\sigma) e^{\lambda \sigma \tau }d\sigma,
\quad \mbox{ for } x \in \Gamma_{1} .
$$
Since $f_{1} \in H_{\Gamma_{0}}^{1}(\Omega) \mbox{ and } f_{4} \in L^{2}(\Gamma_{1}) \times L^{2}(0,1)$, then $z_{1} \in L^{2}(\Gamma_{1})$.

Consequently, knowing  $u$, we may deduce $v$ by  (\ref{solution_v}), $z$ by (\ref{z_formula}) and using (\ref{z1}), we deduce $w = \gamma_{1}(v)$ by (\ref{eqsurj3}).

From  equations (\ref{eqsurj2}) and (\ref{eqsurj3}), $u$ must satisfy:
\begin{equation} \label{equbar}
\lambda (\lambda + a) u - \Delta u = f_{2} +  (\lambda + a)  f_{1}, \quad \mbox{in } \Omega
\end{equation}
with the boundary conditions
\begin{eqnarray}
u = 0, & \mbox{ on } & \Gamma_{0} \label{ubar0} \\
\frac{\partial u}{\partial \nu} = f_{3} - \lambda z(.,0) - \mu z(.,1), &\mbox{on } & \Gamma_{1}. \label{ubar1}
\end{eqnarray}
Using the preceding expression of $z(.,1)$ and the expression of $v$ given by (\ref{solution_v}), we have:
\begin{equation} \label{dudnu}
\frac{\partial u}{\partial \nu} = - \, \left(\lambda^2 + \mu \lambda e^{-\lambda \tau} \right) u + f(x), \quad \mbox{ for  } x \in \Gamma_{1}
\end{equation}
with
$$
f(x) = f_{3}(x) + \lambda f_{1}(x)  - \mu z_{1}(x),  \quad \mbox{ for  } x \in \Gamma_{1} \ .
$$
From the regularity of $f_{3} \,,\, f_{2} \,,\, z_{1}$, we get $f \in L^{2}(\Gamma_{1})$.

The variational formulation of problem (\ref{equbar}), (\ref{ubar0}),(\ref{dudnu}) is to find  $u \in H_{\Gamma_{0}}^{1}(\Omega)$ such that:
\begin{eqnarray}
\int_{\Omega} \lambda (\lambda + a) u \omega + \nabla u \nabla \omega dx &+ &
\int_{\Gamma_{1}} \left(\lambda^2 + \mu \lambda e^{-\lambda \tau}\right) u(\sigma)  \omega(\sigma) d \sigma, \label{varia}\\
&=&
\int_{\Omega}  \left(f_{2} +  (\lambda + a) f_{1}\right) \omega  dx + \int_{\Gamma_{1}} f(\sigma) \omega(\sigma) d \sigma,
\notag
\end{eqnarray}
for any $\omega \in H_{\Gamma_{0}}^{1}(\Omega)$. Since $\lambda >0,\, \mu >0 $, the left hand side of (\ref{varia}) defines a coercive bilinear form on $H_{\Gamma_{0}}^{1}(\Omega)$.
Thus by applying the Lax-Milgram theorem, there exists a unique $u \in H_{\Gamma_{0}}^{1}(\Omega)$ solution of (\ref{varia}). Now, choosing
$\omega \in \mathscr{C}_{c}^{\infty}$, $u$ is a solution of (\ref{equbar}) in the sense of distribution and therefore $u \in H^2(\Omega) \cap H^1_{\Gamma_0}(\Omega)$.
Thus using the Green's formula and exploiting the equation (\ref{equbar}) on $\Omega$, we obtain finally:
$$
\int_{\Gamma_{1}} \left(\lambda^2 + \mu \lambda e^{-\lambda \tau}\right) u(\sigma)  \omega(\sigma) d \sigma  +
 \left\langle\frac{\partial u}{\partial \nu};\omega \right\rangle_{\Gamma_{1}} = \int_{\Gamma_{1}} f(\sigma) \omega(\sigma) d \sigma \ \forall \omega \in H_{\Gamma_{0}}^{1}(\Omega) \ .
$$
So $u \in H^2(\Omega) \cap H^1_{\Gamma_0}(\Omega)$ verifies (\ref{dudnu}) and we recover $u$ and $v$ and thus by (\ref{z_formula}), we obtain $z$ and
finally setting $w = \gamma_{1}(v)$, we have found  $V=(u,v,w,z)^{T}\in \mathscr{D}\left( \mathscr{A}\right) $
solution of $
\left( I -\mathscr{A}\right) V=F $.

Thus from the Lumer-Phillips' theorem, there exists a unique solution $V\in C\left( \mathbb{R}_{+};\mathscr{H}\right) $ of  the shifted problem \eqref{Matrix_problem_Ad}.
This completes the proof of Theorem \ref{existence_u}.
\end{proof}

\begin{remark} \rm \label{opdamp}
According to the above the operator $\mathscr{A}_d$ generates a $C_0$ semigroup of contractions $e^{t\mathscr{A}_d}$ on $\mathscr{H}$.
\end{remark}


\section{Asymptotic behavior}
In this section, we show that if $\xi> \xi^*$, the semigroup $e^{t\mathscr{A}_d}$ decays to the null steady
state with a polynomial decay rate for regular initial data. To obtain this, our technique is based on a frequency domain method and combines a contradiction argument with the multiplier technique to carry out a special analysis for the resolvent. 

\begin{theorem} \label{lr}
Let $\xi > \xi^*$.  Then there exists a constant $C >0$ such that, for all $V_0 \in\mathscr{D}(\mathscr{A}_d)$, the semigroup $e^{t\mathscr{A}_d}$ satisfies the following estimate
\begin{equation}\label{EXPDECEXP3nb}
\left\|e^{t \mathscr{A}_d} V_0\right\|_{\mathscr{H}} \le  \frac{C}{\sqrt{t}} \, \left\Vert V_0 \right\Vert_{\mathscr{D}({\mathscr{A}}_d)},
\forall \; t > 0.
\end{equation}
\end{theorem}
\begin{remark} \rm \label{stabAd}
Let us notice that although the semigroup $e^{t\mathscr{A}_d}$ generates a polynomial stability, we cannot conclude on the stability of the semigroup $e^{t\mathscr{A}}$. 

Indeed let us consider $V_{0} \in \mathscr{D}({\mathscr{A}})$ and $\widetilde{V}_{0} = V_{0} - \mu_{1} V_{0}$, and the two following problems:
\begin{equation*}
\left\{
\begin{array}{lll}
V' &= \mathscr{A} V \\
V_{t=0} &= V_{0}
\end{array}
\right.
and \quad
\left\{
\begin{array}{ll}
\widetilde{V}' &= \mathscr{A}_{d} \widetilde{V} \\
\widetilde{V}_{t=0} &= \widetilde{V}_{0} \ .
\end{array}
\right.
\end{equation*}
Given $V_{0}=(u_{0},u_{1},\gamma_{1}(u_{1}),z_{0})^{T} \in \mathscr{D}({\mathscr{A}})$, the second problem writes in term of \linebreak
$V~=~(u,u_{t},\gamma_{1}(u_{t}),z)^{T}$:
\begin{equation}
\left\{
\begin{array}{ll}
u_{tt}-\displaystyle\frac{1}{1+\mu_{1}} \Delta u + \displaystyle\frac{a+\mu_{1}}{1+\mu_{1}} \, u_{t}=0, & x\in \Omega ,\ t>0 \,,\,\\[0.1cm]
u(x,t)=0, & x\in \Gamma _{0},\ t>0 \,,\,\\[0.1cm]
u_{tt}(x,t)=-  \displaystyle\frac{\partial u}{\partial \nu }(x,t) - \mu z(x,1,t)- \mu u_{t}(x,t), & x\in \Gamma_{1} ,\ t>0 \,,\,\\[0.1cm]
\tau z_{t}(x,\rho ,t)+z_{\rho }(x,\rho ,t)+\mu_{1} \tau z(x,\rho,t)=0, & x\in \Gamma _{1},\rho \in(0,1)\,,\,t>0 \,,\\[0.1cm]
z(x,0,t)=u_{t}(x,t) & x\in \Gamma _{1},\ t>0 \,,\\[0.1cm]
u(x,0)=(1-\mu_{1})u_{0}(x) & x\in \Omega \,,\\
u_{t}(x,0)=(1-\mu_{1})u_{1}(x) & x\in \Omega \,,\\
z(x,\rho ,0)=(1-\mu_{1})f_{0}(x,-\tau \rho ) & x\in \Gamma _{1},\,\rho \in (0,1) \ .
\end{array}%
\right. 
 \label{wave2d}
\end{equation}
We call this problem the ``shifted''  problem.

By Duhamel's formula, we get:
\begin{equation}\label{VandVtilde}
\forall t > 0 \,,\, e^{-\mu_{1}t} V(t) = \frac{-e^{-\mu_{1}t} }{1 + \mu_{1}}V_{0} + \frac{\widetilde{V}(t)}{1 + \mu_{1}} -  \frac{\mu_{1}}{1 + \mu_{1}} \int_{0}^{t}e^{-\mu_{1}(t-s)} \ \widetilde{V}(s) ds \ .
\end{equation}
The first two terms of the right hand side of equation \eqref{VandVtilde} tends to zero as $t$ tends to infinity. So we obtain:
\begin{equation*}
V(t) \simeq \int_{0}^{t}e^{\mu_{1} s} \ \widetilde{V}(s) ds  \ .
\end{equation*}
We only know at this stage that $\Vert \widetilde{V}(s)\Vert_{\mathscr{D}({\mathscr{A}}_d)}$ tends to zero at least as $s^{-1/2}$, and thus $\Vert V(t)\Vert_{\mathscr{D}(\mathscr{A})}$
may tend to zero or blow-up in infinite time. We will illustrate this behavior by numerical examples in 1D in the last section of this work.
\end{remark}
\begin{proof}[Proof of theorem  \ref{lr}]
We will use the following frequency domain theorem for polynomial stability from \cite{borichev} (see also \cite{batty,rao} for weaker variants) of a $C_0$ semigroup of contractions on a Hilbert space:

\begin{lemma}
\label{lemrao}
A $C_0$ semigroup $e^{t{\mathcal L}}$ of contractions on a Hilbert space ${\mathcal H}$ satisfies 
$$||e^{t{\mathcal L}}U_0||_{{\mathcal H}} \leq \frac{C}{t^{\frac{1}{\theta}}} ||U_0||_{{\mathcal D}({\mathcal L})}$$
for some constant $C >0$ and for $\theta>0$ if and only if
\begin{equation}
\rho ({\mathcal L})\supset \bigr\{i \beta \bigm|\beta \in \R \bigr\} \equiv i \R, \label{1.8w} \end{equation}
and \begin{equation}\limsup_{|\beta |\to \infty } \frac{1}{\beta^\theta} \, \|(i\beta I -{\mathcal L})^{-1}\|_{{\mathcal L}({\mathcal H})} <\infty, \label{1.9} 
\end{equation}
where $\rho({\mathcal L})$ denotes the resolvent set of the operator 
${\mathcal L}$.
\end{lemma}

\begin{remark} \rm \label{Non-Compactness}
In view of this theorem we need to identify the spectrum of $\mathscr{A}_d$ lying on the imaginary axis.
Unfortunately, as the embedding of $L^{2}\left(\Gamma_{1},H^{1}(0,1)\right)$ into $L^{2}\left(\Gamma_{1} \times (0,1)\right) = L^{2}\left(\Gamma_{1} \times L^{2}(0,1)\right)$ is not compact,
$\mathscr{A}_{d}$ has not a compact resolvent. Therefore its spectrum $\sigma(\mathscr{A}_{d})$ does not consist only of eigenvalues of $\mathscr{A}_{d}$. 
We have then to show that :
\begin{enumerate}
\item if $\beta$ is a real number, then $i\beta I - {\mathscr  A}_d$ is injective and
\item if $\beta$ is a real number, then $i\beta I - {\mathscr  A}_d$ is surjective.
\end{enumerate}
It is the objective of the two following lemmas.
\end{remark}
First we look at the point spectrum of $\mathscr{A}_d$.
\begin{lemma}\label{pointspectrum}
If $\beta$ is a real number, then $i\beta$  is not an eigenvalue of $\mathscr{A}_d$.
\end{lemma}

\begin{proof} We
will show that the equation 
\begin{equation}\mathscr{A}_d Z = i \beta Z 
\label{eigenAd} 
\end{equation}
with $Z= (u,v,w,z)^T \in \mathscr{D}(\mathscr{A}_d)$ and $\beta \in \R$ has only
the trivial solution.\\
Equation \eqref{eigenAd} writes :
\begin{eqnarray}
(i\beta + \mu_1) u - v &=& 0,\vspace{0.2cm}  \label{eigenAd1}\\
(i \beta + \mu_1) v -  \Delta u + a v  &=& 0, \label{eigenAd2} \\
(i \beta + \mu_1) w + \frac{\partial u}{\partial \nu } + \mu z(.,1) &=& 0, \label{eigenAd3} \\
(i\beta + \mu_1) z + \frac{1}{\tau} z_{\rho}  &=& 0 \ .\label{eigenAd4} \
\end{eqnarray}%

\noindent By taking the inner product of (\ref{eigenAd}) with $Z$ and using \eqref{dissipativeAd}, we get:
\begin{equation}\label{1.7}
\Re
\left(<\mathscr{A}_d Z,Z>_{{\mathscr{H}}} \right)
\leq - \int_\Omega a \left|v(x)\right|^2 \, dx - \left(\frac{\xi}{2\tau} - \frac{\mu}{2}\right) \, \int_{\Gamma_1} \left|z(\sigma,1) \right|^2 \, d\sigma. 
\end{equation}
Thus we firstly obtain that: 
$$v=0 \; \hbox{and} \; z(.,1) = 0.$$ Next, according to \eqref{eigenAd1}, 
we have $v = \left(i\beta + \mu_{1} \right) \, u$. \\
Thus we have $u=0$; since  $w~=~\gamma_1(v)~=~z(.,0)$, we obtain also 
$w = 0$ and $z(.,0)= 0$. Moreover as $z$ satisfies \eqref{eigenAd4}  by integration, we obtain:
$$
z(.,\rho) = z(.,0) \, e^{- \tau (i\beta + \mu_1)\rho}.
$$
But as $z(.,0)= 0$, we finally have $z = 0$.

Thus the only solution of \eqref{eigenAd} is the trivial one.
\end{proof}

Next, we show that $\mathscr{A}_d$ has no continuous spectrum on the imaginary axis.
\begin{lemma}\label{resolventiR}
Let $\xi > \xi^*$. If $\beta$ is a real number, then $i\beta$ belongs to the resolvent set $\rho(\mathscr{A}_d)$ of $\mathscr{A}_d$.
\end{lemma}

\begin{proof}
In view of Lemma \ref{pointspectrum} 
it is enough to show that
$i\beta I - {\mathscr  A}_d$ is surjective.

For $ F=(f_{1},f_{2},f_{3},f_{4})^{T}\in \mathscr{H}$, let $V=(u,v,w,z)^{T}\in \mathscr{D}\left( \mathscr{A}_d\right) $
solution of
\begin{equation*}
\left( i \beta I-\mathscr{A}_d\right) V=F,
\end{equation*}%
which is:
\begin{eqnarray}
(i\beta + \mu_1) u - v &=& f_{1},\vspace{0.2cm}  \label{eqsurj1spec}\\
(i \beta + \mu_1) v -  \Delta u + a v  &=& f_{2}, \label{eqsurj2spec} \\
(i \beta + \mu_1) w + \frac{\partial u}{\partial \nu } + \mu z(.,1) &=& f_{3}, \label{eqsurj3spec} \\
(i\beta + \mu_1) z + \frac{1}{\tau} z_{\rho}  &=& f_{4}.\label{eqsurj4spec} \
\end{eqnarray}%

To find $V=(u,v,w,z)^{T}\in \mathscr{D}\left( \mathscr{A}_d\right) $ solution of the system (\ref{eqsurj1spec}), (\ref{eqsurj2spec}), (\ref{eqsurj3spec}) and (\ref{eqsurj4spec}), we suppose $u$ is determined with the appropriate regularity. Then from  (\ref{eqsurj1spec}), we get:
\begin{equation}
v=(i \beta + \mu_1) u-f_{1} \ .  \label{solution_vspec}
\end{equation}
Therefore, from the compatibility condition on $\Gamma_{1}$, equation (\ref{domainA3}), we determine $z(.,0)$ by:
\begin{equation}
z(x,0) = v( x) =(i\beta + \mu_1) u(x) -f_{1}(x) , \ \mbox{\ for } \, x \in \Gamma_{1}.
\label{z_solutionbis}
\end{equation}
Thus, from  (\ref{eqsurj4spec}), $z$ is the solution of the linear Cauchy problem:
\begin{equation}\label{diffzspec}
\left\{
\begin{array}{ll}
z_{\rho} = \tau\Big(f_{4}(x) - (i\beta + \mu_1) z(x,\rho) \Big), & \mbox{ for } x \in \Gamma_{1} \,,\, \rho \in (0,1), \\
z(x,0)  =(i\beta + \mu_1) u(x) -f_{1}(x).&
\end{array}
\right.
\end{equation}
The solution of the Cauchy problem (\ref{diffzspec}) is given by: $ \mbox{ for } x \in \Gamma_{1} \,,\, \rho \in (0,1)$
\begin{equation}\label{z_formulaspec}
z(x,\rho) = (i\beta + \mu_1) u( x) e^{-(i\beta + \mu_1) \rho \tau} -f_{1}e^{-(i\beta + \mu_1) \rho \tau }+\tau e^{-(i\beta + \mu_1) \rho \tau}\int_{0}^{\rho}f_{4}(x,\sigma) e^{(i\beta + \mu_1) \sigma \tau }d\sigma
\end{equation}
So, we have at the point $\rho = 1$,
\begin{equation} \label{z1spec}
z(x,1) = (i\beta + \mu_1) u(x)  e^{-(i\beta + \mu_1) \tau} + z_{1}(x), \quad \mbox{ for } x \in \Gamma_{1}
\end{equation}
with
$$
z_{1}(x) = -f_{1}e^{-(i\beta + \mu_1) \tau }+\tau e^{-(i\beta + \mu_1)  \tau}\int_{0}^{1}f_{4}(x,\sigma) e^{(i\beta + \mu_1) \sigma \tau }d\sigma,
\quad \mbox{ for } x \in \Gamma_{1} .
$$
Since $f_{1} \in H_{\Gamma_{0}}^{1}(\Omega) \mbox{ and } f_{4} \in L^{2}(\Gamma_{1}) \times L^{2}(0,1)$, then $z_{1} \in L^{2}(\Gamma_{1})$.

Consequently, knowing  $u$, we may deduce $v$ by  (\ref{solution_vspec}), $z$ by (\ref{z_formulaspec}) and using (\ref{z1spec}), we deduce $w = \gamma_{1}(v)$ by (\ref{eqsurj3spec}).

From  equations (\ref{eqsurj2spec}) and (\ref{eqsurj3spec}), $u$ must satisfy:
\begin{equation} \label{equbarspec}
(i\beta + \mu_1) (i\beta + \mu_1 + a) u - \Delta u = f_{2} +  (i\beta + \mu_1 + a)  f_{1}, \quad \mbox{in } \Omega
\end{equation}
with the boundary conditions
\begin{eqnarray}
u = 0, & \mbox{ on } & \Gamma_{0} \label{ubar0spec} \\
\frac{\partial u}{\partial \nu} = f_{3} - (i\beta + \mu_1) z(.,0) - \mu z(.,1), &\mbox{on } & \Gamma_{1}. \label{ubar1spec}
\end{eqnarray}
Using the preceding expression of $z(.,1)$ and the expression of $v$ given by (\ref{solution_vspec}), we have:
\begin{equation} \label{dudnuspec}
\frac{\partial u}{\partial \nu} = - \, \left((i\beta + \mu_1)^2 + \mu \beta e^{-(i\beta + \mu_1) \tau} \right) u + f(x), \quad \mbox{ for  } x \in \Gamma_{1}
\end{equation}
with
$$
f(x) = f_{3}(x) + (i\beta + \mu_1) f_{1}(x)  - \mu z_{1}(x),  \quad \mbox{ for  } x \in \Gamma_{1} \ .
$$
From the regularity of $f_{3} \,,\, f_{2} \,,\, z_{1}$, we get $f \in L^{2}(\Gamma_{1})$.

The variational formulation of problem (\ref{equbarspec}), (\ref{ubar0spec}),(\ref{dudnuspec}) is to find  $u \in H_{\Gamma_{0}}^{1}(\Omega)$ such that \linebreak
for any $\omega \in H_{\Gamma_{0}}^{1}(\Omega)$:
\begin{equation}
\begin{split}
\int_{\Omega} \Big((i\beta  + \mu_1) (i\beta + \mu_1+ a) u \bar{\omega} &+ \nabla u \nabla \bar{\omega}\Big)dx \\
&+ \int_{\Gamma_{1}} \left((i\beta + \mu_1)^2 + \mu (i\beta + \mu_1) e^{-(i\beta + \mu_1) \tau}\right) u(\sigma)  \bar{\omega}(\sigma) d \sigma \\
&=\int_{\Omega}  \left(f_{2} +  (i\beta + \mu_1 + a) f_{1}\right) \bar{\omega}  dx + \int_{\Gamma_{1}} f(\sigma) \bar{\omega}(\sigma) d \sigma
\end{split}
\label{variaspec}
\end{equation}

Multiplying the preceding equation by $-i \beta + \mu_{1}$ leads to, for any $\omega \in H_{\Gamma_{0}}^{1}(\Omega)$:
\begin{equation}
\begin{split}
\int_{\Omega} \Big(|i\beta  + \mu_1|^2 (i\beta + \mu_1+ a) u \bar{\omega} &+ (- i \beta + \mu_1) \nabla u \nabla \bar{\omega}\Big) dx \\
&+ \int_{\Gamma_{1}} |i \beta + \mu_1|^2 \, \left(i\beta + \mu_1 + \mu  e^{-(i\beta + \mu_1) \tau}\right) u(\sigma)  \bar{\omega}(\sigma) d \sigma \\
& = \int_{\Omega}  (- i \beta + \mu_1) \, \left(f_{2} +  (i\beta + \mu_1 + a) f_{1}\right) \bar{\omega}  dx \\
&+ \int_{\Gamma_{1}} (-i \beta + \mu_1) \, f(\sigma) \bar{\omega}(\sigma) d \sigma.
\end{split}
\label{variaspecbis}
\end{equation}
Since $\mu_1 > \mu >0$, the left hand side of (\ref{variaspecbis}) defines a coercive sesquilinear form on $H_{\Gamma_{0}}^{1}(\Omega)$.
Thus by applying the Lax-Milgram theorem, there exists a unique $u \in H_{\Gamma_{0}}^{1}(\Omega)$ solution of (\ref{variaspec}). Now, choosing
$\omega \in \mathscr{C}_{c}^{\infty}$, $u$ is a solution of (\ref{equbarspec}) in the sense of distribution. Using the regularity of $ f_{1}$ and $f_{2}$, we finally have $u \in H^2(\Omega) \cap H^1_{\Gamma_0}(\Omega)$.
Thus using the Green's formula and exploiting the equation (\ref{equbarspec}) on $\Omega$, we obtain finally: $\forall \omega \in H_{\Gamma_{0}}^{1}(\Omega)$,
$$
\int_{\Gamma_{1}} \left((i\beta + \mu_1)^2 + \mu (i\beta + \mu_1) e^{-(i\beta + \mu_1) \tau}\right) u(\sigma)  \omega(\sigma) d \sigma  +
 \left\langle\frac{\partial u}{\partial \nu};\omega \right\rangle_{\Gamma_{1}} = \int_{\Gamma_{1}} f(\sigma) \omega(\sigma) d \sigma  \ .
$$
So $u \in H^2(\Omega) \cap H^1_{\Gamma_0}(\Omega)$ verifies (\ref{dudnuspec}). Then we recover $v$ by equation \eqref{eqsurj1spec} and  by equation \eqref{z_formulaspec}, we obtain $z$.
Finally setting $w = \gamma_{1}(v)$, we have found  $V=(u,v,w,z)^{T}\in \mathscr{D}\left( \mathscr{A}_d\right) $
solution of $
\left( i\beta I -\mathscr{A}_d \right) V=F $.

\end{proof}
The following lemma shows that (\ref{1.9}) holds with $\mathcal{L}=\mathscr{A}_d$ and $\theta=2$.

\begin{lemma}\label{lemresolvent}
The resolvent operator of $\mathscr{A}_d$ satisfies condition \eqref{1.9} for $\theta=2.$
\end{lemma}

\begin{proof}
Suppose that condition \eqref{1.9} is false with $\theta=2$. 
By the Banach-Steinhaus Theorem (see \cite{brezis}), there exists a sequence of real numbers $\beta_n \rightarrow +\infty$ and a sequence of vectors
$Z_n= (u_{n},v_{n},w_n,z_n )^t\in \mathscr{D}(\mathscr{A}_d)$ with $\|Z_n\|_{\mathscr{H}} = 1$ such that 
\begin{equation}
|| \beta_n^2 (i \beta_n I - \mathscr{A}_d)Z_n||_{\mathscr{H}} \rightarrow 0\;\;\;\; \mbox{as}\;\;\;n\rightarrow \infty, 
\label{1.12} \end{equation}
i.e., 
\begin{equation}\beta_n \left((i \beta_n + \mu_{1})u_{n} - v_{n}\right) \equiv f_{n}\rightarrow 0 \;\;\; \mbox{in}\;\; H^1_{\Gamma_0} (\Omega), 
\label{1.13}
\end{equation}
\begin{equation}
\beta_n \left( i \beta_n
v_{n} - \Delta u_n + (\mu_{1} + a) v_n \right) \equiv g_{n} \rightarrow 0 \;\;\;
\mbox{in}\;\; L^2(\Omega),
\label{1.13b} \end{equation}
 \begin{equation}
\beta_n \, \left( (i \beta_n + \mu_{1} ) w_{n} + \frac{\partial u_n}{\partial \nu} + \mu z_{n}(.,1) \right) \equiv h_{n} \rightarrow 0 \;\;\;
\mbox{in}\;\; L^2(\Gamma_1), 
\label{1.14b} \end{equation}
\begin{equation}\label{1.14} 
\beta_n \left( (i\beta_n + \mu_{1}) z_n + \frac{1}{\tau} \partial_\rho z_n\right) \equiv k_n \rightarrow 0 \; \; \; \mbox{in} \;\; L^2(\Gamma_1 \times (0,1))
\end{equation}
since $\beta_n \leq\beta_n^2$.

Our goal is to derive from \eqref{1.12} that $||Z_n||_{\mathscr{H}}$ converges to zero, thus there is a contradiction. 

We first notice that we have
\begin{equation}
|| \beta_n^2 (i \beta_n I - \mathscr{A}_d)Z_n||_{\mathscr{H}} \ge |\Re \left(\langle \beta_n^2(i\beta_n I - \mathscr{A}_d)Z_n, Z_n\rangle_{\mathscr{H}} \right)|. 
\label{1.15}
\end{equation}
Then, by \eqref{1.7} and \eqref{1.12}, 
\begin{equation}
\beta_n \,v_{n} \rightarrow 0, \;\;\;
\mbox{in}\;\; L^2(\Omega),\quad  \beta_n \,z_n (.,1) \rightarrow 0, \;\;\;
\mbox{in}\;\; L^2(\Gamma_1),
\label{1.16a}\end{equation}
and 
\begin{equation}
u_n \rightarrow 0, \, \Delta u_n \rightarrow 0 \;\;\;
\mbox{in}\;\; L^2(\Omega) \Rightarrow u_n \rightarrow 0 \;\;\;
\mbox{in}\;\; H^1_{\Gamma_0}(\Omega)
\label{1.16} \ .
\end{equation}

This further leads, by (\ref{1.14b}) and the trace theorem, to
\begin{equation}
w_n \rightarrow 0 \;\;\;
\mbox{in}\;\; L^2(\Gamma_1).
\label{1.17a} \end{equation}
Moreover, since $Z_n\in{\mathcal D}(\mathscr{A}_d)$, we have, by (\ref{1.17a}),
\begin{equation}
z_n (.,0) \rightarrow 0 \;\;\;
\mbox{in}\;\; L^2(\Gamma_1).
\label{4.50b}\end{equation}
We have
\begin{equation}
\label{4.50c}
z_n(.,\rho) = z_n (.,0) \, e^{-(i\beta_n + \mu_1) \tau \rho} + \int_0^\rho e^{- (i \beta_n + \mu_1) \tau (\rho - s)} \, \frac{ \tau k_n(s)}{\beta_n} \, ds.
\end{equation}

Which implies, according to \eqref{4.50c}, \eqref{4.50b} and \eqref{1.14}, that 
$$
z_n \rightarrow 0 \; \; \mbox{in}\;\; L^2(\Gamma_1 \times (0,1))
$$
and clearly contradicts $\left\|Z_n\right\|_{\mathscr{H}}=1$.
\end{proof}

The two hypotheses of Lemma \ref{lemrao} are proved by Lemma \ref{lemresolvent} and Lemma \ref{resolventiR}. Then (\ref{EXPDECEXP3nb}) holds. The proof of Theorem \ref{lr} is then finished.
\end{proof}


\section{Changing the damping law}
Let us consider now the same system as \eqref{ondes} but with a Kelvin-Voigt damping. 
The system is given by:

\begin{equation}
\left\{
\begin{array}{ll}
u_{tt}-\Delta u - a \, \Delta u_{t}=0, & x\in \Omega ,\ t>0 \,,\,\\[0.1cm]
u(x,t)=0, & x\in \Gamma _{0},\ t>0 \,,\,\\[0.1cm]
u_{tt}(x,t)=- \displaystyle\frac{\partial u}{\partial \nu }(x,t) - a \, \frac{\partial u_t}{\partial \nu} - \mu u_{t}(x,t-\tau ) & x\in \Gamma_{1},\ t>0 \,,\,\\[0.1cm]
u(x,0)=u_{0}(x) & x\in \Omega \,,\,\\[0.1cm]
u_{t}(x,0)=u_{1}(x) & x\in \Omega \,,\, \\[0.1cm]
u_{t}(x,t-\tau)=f_{0}(x,t-\tau ) & x\in\Gamma_{1},\, t\in (0,\tau).
\end{array}%
\right.  \label{ondeskv}
\end{equation}
Which, as above, is equivalent to:
\begin{equation}
\left\{
\begin{array}{ll}
u_{tt}-\Delta u - a \, \Delta u_{t} =0, & x\in \Omega ,\ t>0 \,,\,\\[0.1cm]
u(x,t)=0, & x\in \Gamma _{0},\ t>0 \,,\,\\[0.1cm]
u_{tt}(x,t)=-  \displaystyle\frac{\partial u}{\partial \nu }(x,t) - a \displaystyle\frac{\partial u_t}{\partial \nu }(x,t) - \mu z(x,1,t), & x\in \Gamma_{1} ,\ t>0 \,,\,\\[0.1cm]
\tau z_{t}(x,\rho ,t)+z_{\rho }(x,\rho ,t)=0, & x\in \Gamma _{1},\rho \in(0,1)\,,\,t>0 \,,\\[0.1cm]
z(x,0,t)=u_{t}(x,t) & x\in \Gamma _{1},\ t>0 \,,\\[0.1cm]
u(x,0)=u_{0}(x) & x\in \Omega \,,\\
u_{t}(x,0)=u_{1}(x) & x\in \Omega \,,\\
z(x,\rho ,0)=f_{0}(x,-\tau \rho ) & x\in \Gamma _{1},\,\rho \in (0,1) \ .
\end{array}%
\right. 
 \label{wave2kv}
\end{equation}

Let the operator $\mathscr{A}_{kv}$ defined by:
\begin{equation*}
\mathscr{A}_{kv}\left(
\begin{array}{c}
u\vspace{0.2cm} \\
v\vspace{0.2cm} \\
w\vspace{0.2cm} \\
z
\end{array}
\right) =\left(
\begin{array}{c}
\displaystyle v\vspace{0.2cm} \\
\displaystyle \Delta u+ a \,\Delta v\vspace{0.2cm} \\
\displaystyle -\frac{\partial u}{\partial \nu } - a \frac{\partial v}{\partial \nu } - \mu z\left( .,1\right) \vspace{0.2cm}\\
\displaystyle -\frac{1}{\tau }z_{\rho }
\end{array}%
\right).
\end{equation*}%

The domain of $\mathscr{A}_{kv}$ is the set of $V=(u,v,w,z)^{T}$ such that:
\begin{eqnarray}
(u,v,w,z)^{T}\in \left(H_{\Gamma_{0}}^{1} (\Omega)\cap H^{2}(\Omega)\right)\times H^{1}_{\Gamma_0} (\Omega)\times L^{2}(\Gamma_{1})\times L^{2}\left(\Gamma_{1};H^{1}(0,1)\right), \, 
&\\
\displaystyle\frac{\partial v}{\partial \nu} \in L^2(\Gamma_1),& \label{domainA1kv}\vspace*{0.3cm} &\\\label{domainA2kv}
w=\gamma_{1}(v) = \,z(.,0) \text{ on }\Gamma_{1}.&& \label{domainA3kv}
\end{eqnarray}
\textbf{Notations:}\\
For $c \, \in \R$, we define:
\begin{equation}\label{Cbeta}
 C_{\Omega}(c) = \inf_{u \in H_{\Gamma _{0}}^{1}(\Omega )}\frac{\Vert\nabla u\Vert_{2}^{2} + c \Vert u \Vert_{2,\Gamma_{1}}^{2}}{\Vert u \Vert_{2}^{2}}
\end{equation}
$ C_{\Omega}(c)$ is the first eigenvalue of the operator $-\Delta$ under the Dirichlet-Robin boundary conditions:
\begin{equation}\label{eigen-dirichlet-robin}
\left\{\begin{array}{ll} 
u(x)=0, &  x \in \Gamma_{0} \\
\displaystyle\frac{\partial u}{\partial \nu}(x) + c  u(x) = 0 & x \in \Gamma_{1} \ .
\end{array}
\right.
\end{equation}
From Kato's perturbation theory \cite{Kato95} (see also \cite[Theorem 1.3.1]{K2010}), $C_{\Omega}(c)$ is a continuous increasing function. 
From Poincaré's inequality and the continuity of the trace operator $\gamma_{1}$, we have $C_{\Omega}(0) > 0$ and 
$C_{\Omega}(c) \rightarrow -\infty $ as $c\rightarrow -\infty $. Thus it exists a unique 
$c^{\star} < 0$ such that:
\begin{equation}\label{defcstar}
C_{\Omega}(c^{\star}) = 0 \ .
\end{equation}

In the following, we fix  $\xi = \mu \tau$ in the norm \eqref{energynorm}. We will see in the next result why this choice is well adapted.
\begin{theorem}\label{existencekv}
Suppose that $a\mbox{ and }\mu$ satisfy the following assumption:
\begin{equation} 
\label{condex}
\mu < |c^{\star}| a. 
\end{equation} 
 Then, the operator $\mathscr{A}_{kv}$ generates a $C_0$ semigroup of contractions on $\mathscr{H}$. We have, in particular, 
if $V_{0}\in \mathscr{H}$, then there
exists a unique solution $V\in C\left( \mathbb{R}_{+};\mathscr{H}%
\right) $ of problem (\ref{ondeskv}). Moreover, if $V_{0}\in %
\mathscr{D}\left( \mathscr{A}_{kv} \right) $, then
\begin{equation*}
V\in C\left( \mathbb{R}_{+};\mathscr{D}\left( \mathscr{A}_{kv}\right) \right)
\cap C^{1}\left( \mathbb{R}_{+};\mathscr{H}\right) .
\end{equation*}
\end{theorem}
\begin{proof} To prove Theorem \ref{existencekv}, we use again the
semigroup approach and the Lumer-Phillips' theorem.  

For this purpose, we show firstly that the operator $\mathscr{A}_{kv}$ is
dissipative. 

Indeed, let $V=(u,v,w,z)^{T}\in \mathscr{D}\left(\mathscr{A}_{kv}\right) $. By definition of the operator $\mathscr{A}_{kv}$ and
the scalar product of $\mathscr{H}$, we have:
\begin{equation*}
\begin{array}{ll}
\displaystyle \left\langle \mathscr{A}_{kv}V,V\right\rangle_{\mathscr{H}} =& \displaystyle \int_{\Omega}\nabla u.\nabla v dx + \int_{\Omega }v\left( \Delta u+ a \Delta v\right) dx  \\
\displaystyle &+ \displaystyle \int_{\Gamma_{1}}w\left( -\frac{\partial u}{\partial \nu }- a \frac{\partial v}{\partial \nu }-\mu z\left(\sigma,1\right)
\right) d\sigma -\dfrac{\xi }{\tau }\int_{\Gamma _{1}}\int_{0}^{1}zz_{\rho}d\rho d\sigma . \
\end{array}
\end{equation*}
Applying Green's formula and the compatibility condition $w = \gamma_{1}(v)$, we obtain: 
\begin{equation}\label{dissipative1kv}
\left\langle \mathscr{A}_{kv}V,V\right\rangle_{\mathscr{H}} = -\mu\int_{\Gamma_{1}}z\left(\sigma,1\right) wd\sigma - a \int_{\Omega }\left\vert \nabla
v\right\vert ^{2}dx-\frac{\xi }{\tau }\int_{\Gamma _{1}}\int_{0}^{1}z_{\rho
}zd\rho dx. 
\end{equation}
But we have:
\begin{eqnarray}
\frac{\xi }{\tau }\int_{\Gamma _{1}}\int_{0}^{1}z_{\rho }z(\sigma,\rho,t) \,
d\rho \, d\sigma &=&\frac{\xi }{2\tau}\int_{\Gamma _{1}} \int_{0}^{1}\frac{%
\partial }{\partial \rho }z^{2}(\sigma,\rho ,t) \, d\rho \, d\sigma  \notag
\\
&=&\frac{\xi }{2\tau}\int_{\Gamma_{1}}\left(
z^{2}(\sigma,1,t)-z^{2}(\sigma,0,t) \right) d\sigma \;.
\label{Energy_second_equationkv}
\end{eqnarray}
Thus from the compatibility condition (\ref{domainA3}), we get: 
\begin{eqnarray*}
- \, \frac{\xi }{\tau }\int_{\Gamma _{1}}\int_{0}^{1}z_{\rho }z \,d\rho \, d\sigma &=&\frac{\xi }{2\tau}\int_{\Gamma_{1}}\left(
v^{2} -z^{2}(\sigma,1,t) \right) d\sigma \;.
\end{eqnarray*}
Therefore equation \eqref{dissipative1kv} becomes:
\begin{equation}\label{dissipative2kv}
\begin{array}{ll}
\displaystyle \left\langle \mathscr{A}_{kv}V,V\right\rangle_{\mathscr{H}} &=\displaystyle- a \int_{\Omega }\left\vert \nabla v\right\vert ^{2}dx 
\displaystyle  + \frac{\xi}{2 \tau} \int_{\Gamma _{1}}v^{2}d\sigma 
-\frac{\xi }{2 \tau }\int_{\Gamma _{1}}\int_{0}^{1}z^{2}(\sigma,1,t) d\sigma \\
~&\displaystyle -\mu\int_{\Gamma_{1}}v(\sigma,t) z\left(\sigma,1\right) d\sigma  \ .
\end{array}
\end{equation}
To treat the last term in the preceding equation, Young's inequality gives: 
\begin{equation}\label{young}
-\displaystyle \int_{\Gamma_{1}}v(\sigma,t) z\left(\sigma,1\right) d\sigma
\leq \displaystyle  \frac{1}{2} \int_{\Gamma_{1}}z^{2}\left(\sigma,1\right) d\sigma+  \displaystyle  \frac{1}{2} \int_{\Gamma_{1}} v^{2}(\sigma,t) d\sigma \ .
\end{equation}
Therefore, we firstly get:
\begin{equation*}
\displaystyle \left\langle \mathscr{A}_{kv}V,V\right\rangle_{\mathscr{H}} + \displaystyle a \int_{\Omega }\left\vert \nabla v\right\vert^{2}dx 
- \displaystyle  \left(\frac{\xi}{2 \tau} + \frac{\mu}{2}\right)\int_{\Gamma _{1}}v^{2}d\sigma +
\displaystyle \left(\frac{\xi }{2 \tau} - \frac{\mu}{2}\right)\int_{\Gamma _{1}}z^{2}(\sigma,1,t) d\sigma  \leq 0 \ .
\end{equation*}
At this point, as $\xi = \mu \tau$,  the previous inequality becomes:
\begin{equation*}
\displaystyle \left\langle \mathscr{A}_{kv}V,V\right\rangle_{\mathscr{H}} + \displaystyle a \int_{\Omega }\left\vert \nabla v\right\vert^{2}dx 
- \displaystyle  \mu \int_{\Gamma _{1}}v^{2}d\sigma  \leq 0 \ .
\end{equation*}
Denoting now $c = - \displaystyle \frac{\mu}{a}$, we get:
\begin{equation*}
\displaystyle \left\langle \mathscr{A}_{kv}V,V\right\rangle_{\mathscr{H}} + \displaystyle a \left(\int_{\Omega }\left\vert \nabla v\right\vert^{2}dx 
+ c \int_{\Gamma _{1}}v^{2}d\sigma \right) \leq 0 \ .
\end{equation*}
By definition \eqref{Cbeta}, we thus get:
\begin{equation}\label{dissipativekv}
\displaystyle \left\langle \mathscr{A}_{kv}V,V\right\rangle_{\mathscr{H}} + \displaystyle a  \, C_{\Omega}(c)\Vert v\Vert_{2}^{2} \leq 0 \ .
\end{equation}
From assumption \eqref{condex}, $C_{\Omega}(c) > 0$. This inequality proves that the operator $\mathscr{A}_{kv}$is dissipative.
To show that  $\lambda I -\mathscr{A}_{kv}$ is surjective for all $\lambda > 0$, we easily adapt the proof of Theorem \ref{existence_u}.

The proof of Theorem \ref{existencekv}, follows from the Lumer-Phillips' theorem.
\end{proof}

Moreover the semigroup operator $e^{t\mathscr{A}_{kv}}$ is exponential stable on $\mathscr{H}$. We have the following result.
\begin{theorem} \label{lrkv}
Suppose that the assumption \eqref{condex} is satisfied. Then, there exist $C,\omega > 0$ such that for all $t > 0$ we have 
$$
\left\|e^{t \mathscr{A}_{kv}}\right\|_{{\mathcal L} (\mathscr{H})} \leq C \, e^{-\omega t}.
$$ 
\end{theorem}
\begin{remark} \rm
We note here that, in this case where the damping operator is sufficiently unbounded for controlling the delay one, 
we obtain the exponential stability result. 

Note that without internal damping (i.e. if $a = 0$), the previous model is destabilized for arbitrarily small delays for 
every value $\mu > 0$ , see \cite{Dat97}. 
Thus, the internal damping $-a\Delta u_{t}$ makes the system robust with respect to time delays in the boundary condition 
if the coefficient $a$ is sufficiently large with respect to $\mu$.
\end{remark}
\begin{remark} \rm In the recent work of Nicaise and Pignotti, \cite{NP2011}, they studied the existence and stability of a problem closely related to
problem \eqref{ondeskv}. 
They obtain a slightly different condition to ensure the existence and the stability. Namely, let us define $C_{P}$, a sort of Poincaré's constant, by:
$$
 C_{P} = \sup_{u \in H_{\Gamma _{0}}^{1}(\Omega )}\frac{\Vert u \Vert_{2,\Gamma_{1}}^{2}}{\Vert\nabla u\Vert_{2}^{2}} \ .
$$
By using the semigroup approach and the Lumer-Phillips' theorem they proved the global existence of the solution
and by using an observability inequality they proved the exponential stability under the condition:
$$
 \mu C_{P}< a  \ .
$$
A simple argument shows that $\frac{1}{C_{P}} \geq |c^{\star}|$. Thus in this work, we obtain a better bound for the coefficient $a$ than the one obtained by Nicaise and Pignotti
since we proved the existence and the exponential stability for : 
$$\frac{\mu}{ |c^{\star}|} < a \leq {\mu}{ C_{P}}.$$

\end{remark}

\begin{proof}[Proof of theorem  \ref{lrkv}]

We will employ the following frequency domain theorem for exponential stability from \cite{pruss} of a $C_0$ semigroup of contractions on a Hilbert space:

\begin{lemma}
\label{lemraokv}
A $C_0$ semigroup $e^{t{\mathcal L}}$ of contractions on a Hilbert space ${\mathcal H}$ satisfies, for all $t >0$, 
$$||e^{t{\mathcal L}}||_{{\mathcal L}({\mathcal H})} \leq C e^{-\omega t}$$
for some constant $C, \omega >0$  if and only if
\begin{equation}
\rho ({\mathcal L})\supset \bigr\{i \beta \bigm|\beta \in \R \bigr\} \equiv i \R, \label{1.8wkv} \end{equation}
and
 \begin{equation}\limsup_{|\beta |\to \infty }  \|(i\beta I -{\mathcal L})^{-1}\|_{{\mathcal L}({\mathcal H})} <\infty, \label{1.9kv} 
\end{equation}
where $\rho({\mathcal L})$ denotes the resolvent set of the operator 
${\mathcal L}$.
\end{lemma}

The proof of Theorem \ref{lrkv} is based on the following lemmas.

For the same reason as before (see remark \eqref{Non-Compactness}), we have to show that there is no eigenvalue lying on the imaginary axis
and that $\mathscr{A}_{kv}$ has no continuous spectrum on the imaginary axis.

We first we look at the point spectrum.

\begin{lemma}\label{condspkv}
If $\beta$ is a real number, then $i\beta$  is not an eigenvalue of $\mathscr{A}_{kv}$.
\end{lemma}
\begin{proof} We
will show that the equation 
\begin{equation}\mathscr{A}_{kv} Z = i \beta Z 
\label{eigenkv} 
\end{equation}
with $Z= (u,v,w,z)^T \in \mathscr{D}(\mathscr{A}_{kv})$ and $\beta \in \mathbb{R}$ has only
the trivial solution.
System \eqref{eigenkv} writes:
\begin{eqnarray}
v &=& i \beta u \label{eigen1}\\
\Delta u+ a \,\Delta v &=& i \beta v \label{eigen2}\\
-\frac{\partial u}{\partial \nu } - a \frac{\partial v}{\partial \nu } - \mu z\left( .,1\right) &=& i \beta w \label{eigen3}\\
-\frac{1}{\tau }z_{\rho } &=& i \beta z \label{eigen4}
\end{eqnarray}
\noindent Denoting now $c = - \displaystyle \frac{\mu}{a}$, by taking the inner product of (\ref{eigenkv}) with $Z$, ,
using  the inequality \eqref{dissipativekv} we get: 
\begin{equation}\label{1.7kv}
\Re
\left(<\mathscr{A}_{kv} Z,Z>_{{\mathscr{H}}} \right)
\leq - a  \, C_{\Omega}(c)\Vert v\Vert_{2}^{2} \ .
\end{equation}
From assumption \eqref{condex}, $C_{\Omega}(c) > 0$ and thus we obtain that $v=0$. 

Next, 
we have $w = \gamma_1 (v)$, we also have $w = 0$. Moreover as we have $w = z(.,0)$, we get: $z(.,0) = 0$.

From \eqref{eigen3} we also have  $z(.,1) = 0$.

As $z$ satisfies  \eqref{eigen4}, we get the following  identity: 

$$
z(.,\rho) = e^{-i \beta \tau \rho} z(.,0).
$$
Thus the only solution of \eqref{eigenkv} is the trivial one.
\end{proof}

By the same way, as in Lemma \ref{resolventiR}, we show that $\mathscr{A}_{kv}$ has no continuous spectrum on the imaginary axis.

\begin{lemma}\label{resolventiRkv}
If $\lambda$ is a real number, then $i\lambda$ belongs to the resolvent set $\rho(\mathscr{A}_{kv})$ of $\mathscr{A}_{kv}$.
\end{lemma}
\begin{lemma}\label{lemresolventkv}
The resolvent operator of $\mathscr{A}_{kv}$ satisfies condition \eqref{1.9kv}.
\end{lemma}

\begin{proof}
Suppose that condition \eqref{1.9kv} is false. 
By the Banach-Steinhaus Theorem (see \cite{brezis}), there exists a sequence of real numbers $\beta_n \rightarrow +\infty$ and a sequence of vectors
$Z_n= (u_{n},v_{n},w_n,z_n )^T \in \mathscr{D}(\mathscr{A}_{kv})$ with $\|Z_n\|_{\mathscr{H}} = 1$ such that 
\begin{equation}
|| (i \beta_n I - \mathscr{A}_{kv})Z_n||_{\mathscr{H}} \rightarrow 0\;\;\;\; \mbox{as}\;\;\;n\rightarrow \infty, 
\label{1.12kv} \end{equation}
i.e., 
\begin{equation}\left(i \beta_n u_{n} - v_{n}\right) \equiv f_{n}\rightarrow 0 \;\;\; \mbox{in}\;\; H^1_{\Gamma_0} (\Omega), 
\label{1.13kv}\end{equation}
 \begin{equation}
  \left( i \beta_n
v_{n} - \Delta u_n  - a \Delta v_n \right) \equiv g_{n} \rightarrow 0 \;\;\;
\mbox{in}\;\; L^2(\Omega),
\label{1.13bkv} \end{equation}
 \begin{equation}
 \left( i \beta_n w_{n} + \frac{\partial u_n}{\partial \nu} + a \frac{\partial v_n}{\partial \nu} + \mu z_{n}(.,1) \right) \equiv h_{n} \rightarrow 0 \;\;\;
\mbox{in}\;\; L^2(\Gamma_1), 
\label{1.14bkv} \end{equation}
\begin{equation}\label{1.14kv} 
\left( i\beta_n  z_n + \frac{1}{\tau} \partial_\rho z_n\right) \equiv k_n \rightarrow 0 \; \; \; \mbox{in} \;\; L^2(\Gamma_1 \times (0,1)).
\end{equation}

Our goal is to derive from \eqref{1.12kv} that $||Z_n||_{\mathscr{H}}$ converges to zero, thus there is a contradiction. 

We first notice that we have
\begin{equation}
|| (i \beta_n I - \mathscr{A}_{kv})Z_n||_{\mathscr{H}} \ge |\Re \left(\langle (i\beta_n I - \mathscr{A}_{kv})Z_n, Z_n\rangle_{\mathscr{H}} \right)|. 
\label{1.15kv}
\end{equation}
Thus by \eqref{1.7kv} and \eqref{1.12kv}, 
$$v_{n} \rightarrow 0  \mbox{ in } L^{2}(\Omega) \ . $$
From \eqref{1.13kv}, 
$$u_n \rightarrow 0 \mbox{ in }  L^{2}(\Omega) \ .$$
But we also have, 
\begin{equation*} \label{unkv}
\Delta u_n \rightarrow 0 \mbox{ in }  L^{2}(\Omega) \ 
 \mbox { and }
\Delta  v_n \rightarrow 0 \mbox{ in }  L^{2}(\Omega) \ .
\end{equation*}
Thus we firstly obtain: 
\begin{equation} \label{unvnkv}
u_n \rightarrow 0 \mbox{ in } H^1_{\Gamma_0} (\Omega) \ 
 \mbox { and }
v_n \rightarrow 0 \mbox{ in }  H^1_{\Gamma_0} (\Omega) \ .
\end{equation}
By  the trace theorem, we have:
\begin{equation} \label{wnkv}
w_n = \gamma(v_{n})\rightarrow 0 \mbox{ in }L^2(\Gamma_1) \ .
\end{equation}
Moreover, since $Z_n\in{\mathcal D}(\mathscr{A}_{kv})$, $w_{n} = z_n (.,0)$. Thus we get:
\begin{equation}\label{z0kv}
z_n (.,0) \rightarrow 0 \mbox{ in }\;\; L^2(\Gamma_1) \ .
\end{equation}
Now from \eqref{1.14bkv}, we also have:
\begin{equation}\label{z1kv}
z_n (.,1) \rightarrow 0 \mbox{ in }\;\; L^2(\Gamma_1) \ .
\end{equation}
As we have the following identity:
$$
z_n(.,\rho) = z_n (.,0) \, e^{- i \tau \beta_n \rho} + \tau \, \int_0^\rho e^{- i \tau\beta_n (\rho - s)}k_n (.,s) \, ds
$$
according to \eqref{1.14kv} \eqref{z0kv} we finally have:
\begin{equation} \label{znkv}
z_n \rightarrow 0 \; \; \mbox{in}\;\; L^2(\Gamma_1 \times (0,1)) \ .
\end{equation}
Identities \eqref{unvnkv},\eqref{wnkv} and \eqref{znkv} clearly contradicts the fact that:
$$ \forall \ n  \in \N \,,\, \left\|Z_n\right\|_{\mathscr{H}}=1 \ . $$
\end{proof}
The two hypotheses of Lemma \ref{lemraokv} are proved. The proof of Theorem \ref{lrkv} is then finished.
\end{proof}
\section{Comments and numerical illustrations}
To illustrate numerically the results presented in this paper, we present numerical simulations for 
problem \eqref{ondes} and for the Kelvin-Voigt damping, namely problem \eqref{ondeskv}, in 1D.
So let us consider $\Omega = (0,1) \,,\, \Gamma_{0} = \{0\}, \Gamma_{1} = \{1\}$.

To solve numerically problem \eqref{ondes} (resp. problem \eqref{ondeskv}), we have to consider its equivalent formulation, namely problem \eqref{wave2}
(resp. problem \eqref{wave2kv}), which writes in the present case:
\begin{equation}
\left\{
\begin{array}{ll}
u_{tt}-u_{xx} + a \, u_{t} =0, & x\in (0,1) ,\ t>0 \,,\,\\[0.1cm]
u(0,t)=0, &  t>0 \,,\,\\[0.1cm]
u_{tt}(1,t)=-  u_{x}(1,t) - \mu z(x,1,t), &  t>0 \,,\,\\[0.1cm]
\tau z_{t}(1,\rho ,t)+z_{\rho }(1,\rho ,t)=0, & \rho \in(0,1)\,,\,t>0 \,,\\[0.1cm]
z(1,0,t)=u_{t}(1,t) &  t>0 \,,\\[0.1cm]
u(x,0)=u_{0}(x) & x\in (0,1)\,,\\
u_{t}(x,0)=u_{1}(x) & x\in (0,1) \,,\\
z(1,\rho ,0)=f_{0}(1,-\tau \rho ) & \rho \in (0,1) \ .
\end{array}%
\right. 
 \label{wave1D}
\end{equation}
A 1D formulation of the ``shifted'' problem \eqref{wave2d} as well as the Kelvin-Voigt damping problem, problem \eqref{wave2kv} is
of the same type. 

As the stability result that we have presented in this work, namely Theorem \ref{lr},
is a stability result  for the ``shifted'' problem  \eqref{wave2d}, we have to perform numerical simulations for both problems: 
the original one, problem \eqref{ondes} and the ``shifted'' problem  \eqref{wave2d}.

For this sake, to avoid a CFL condition between the mesh size and the time step, we decided to discretize the different problems 
by implicit first order in time, and finite difference method in space.
For every simulations the numerical parameters are the following:
$$
\begin{array}{ll} \tau = 2 \,,\, \xi = 2 \xi^{\star}\,,\, \Delta x = \frac{1}{20}\,,\, \Delta \rho = \frac{1}{20} \,,\, \Delta t = 0.1\\
u_{0}(x) = u_{1}(x) = x e^{10 x} \,,\,  f_{0} (1,\rho)= e^{\rho} e^{10} \ .
\end{array}
$$
For every time $t>0$, we denote $E(t) = \left\Vert \big(u(.,t),u_{t}(.,t),u_{t}(1,t),z(1,.,t)\big)^{T}\right\Vert_{\mathscr{H}}$.
The choice of $u_{0}\,,u_{1},f_{0}$ ensures a large initial energy.

In Figure \ref{wave1mu} and Figure \ref{wave1a}, we present the resulting simulation for the original problem and the ``shifted'' one.
\begin{figure}[H]
\subfigure[Original problem \label{origmu}]{\includegraphics[scale=0.4]{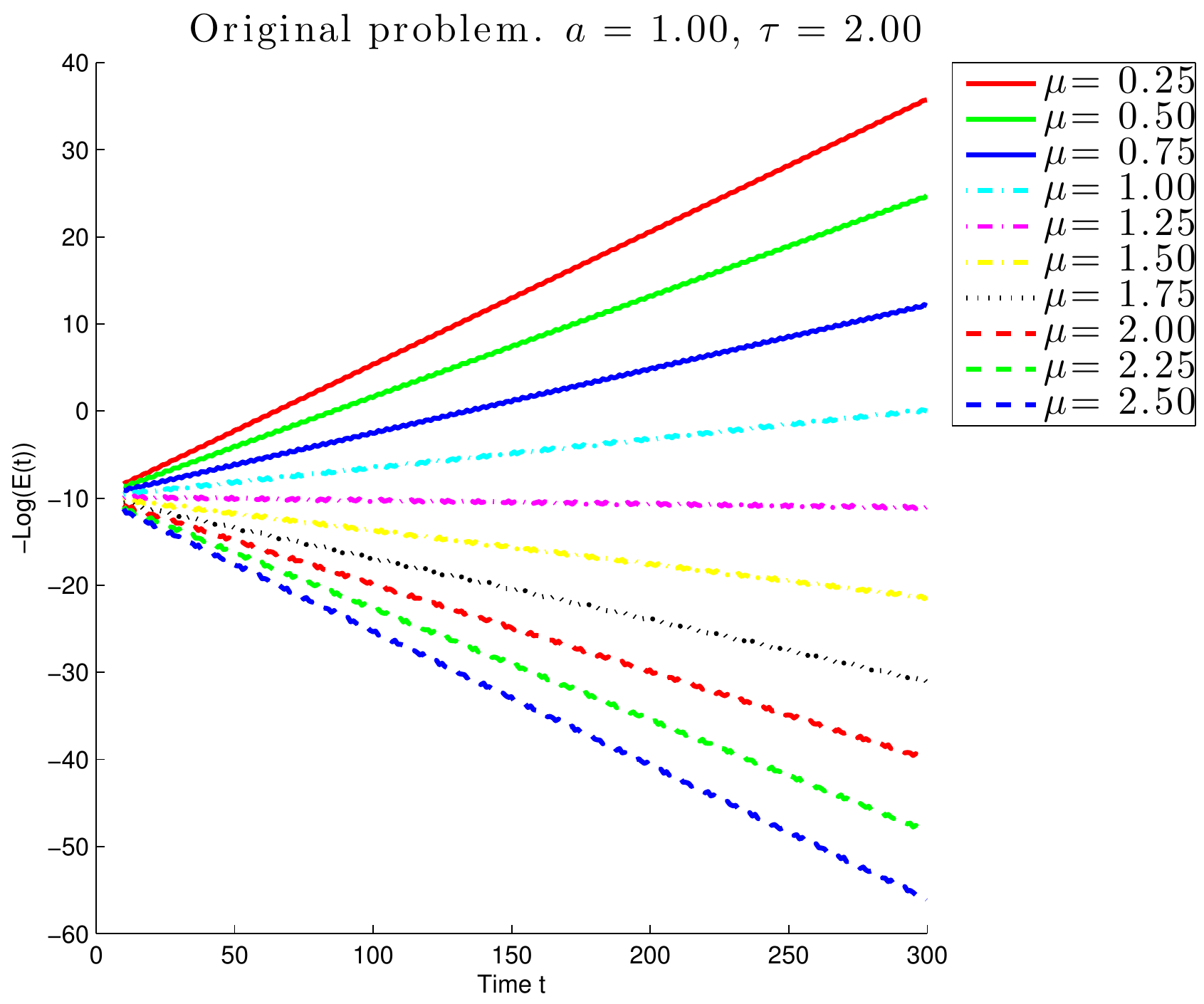}}
\subfigure[Shifted problem \label{shiftedmu}]{\includegraphics[scale=0.4]{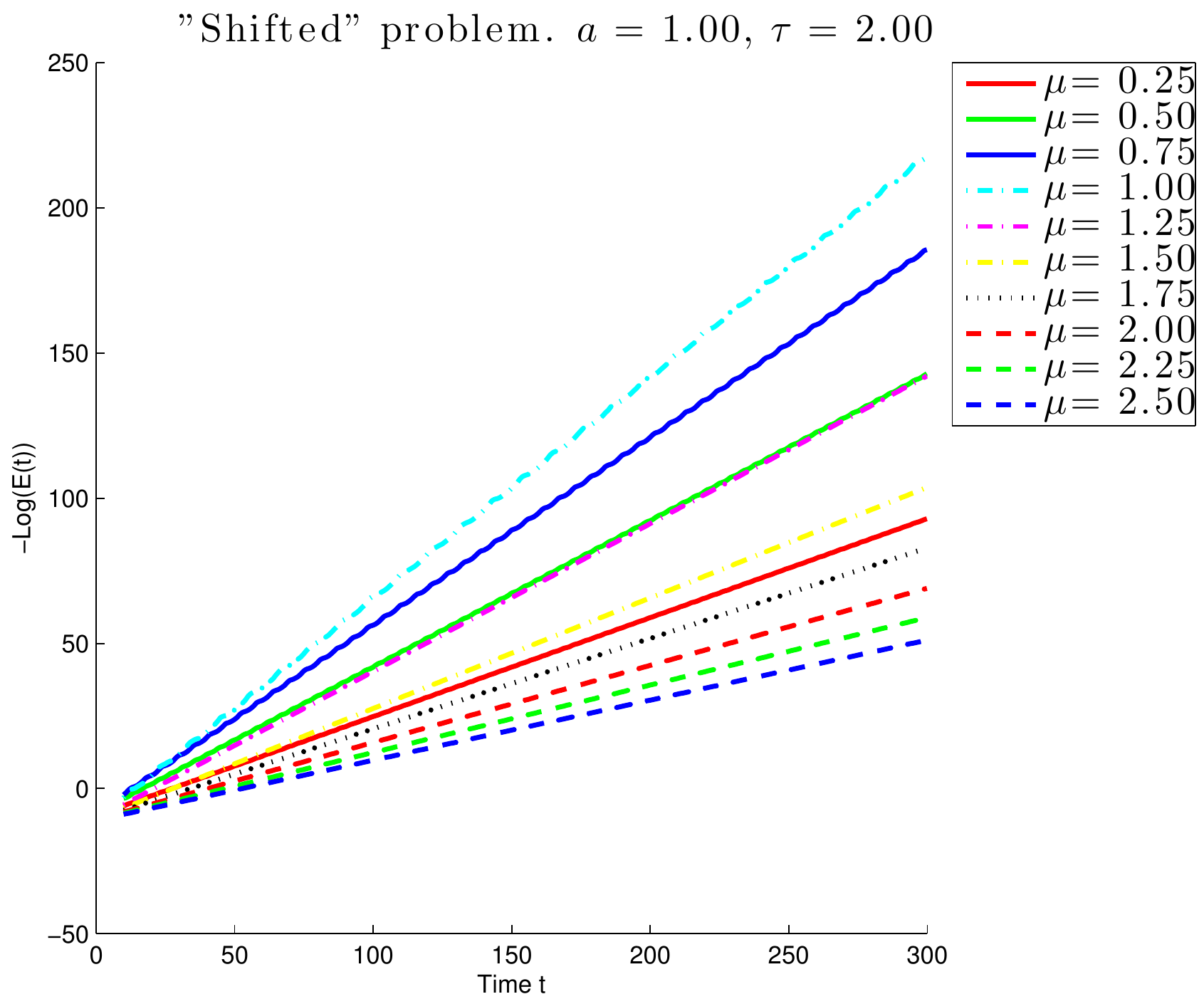}}
\caption{Energy (in -log scale) versus time: influence of $\mu$. \label{wave1mu}}
\end{figure}
\begin{figure}[H]
\subfigure[Original problem \label{origa}]{\includegraphics[scale=0.4]{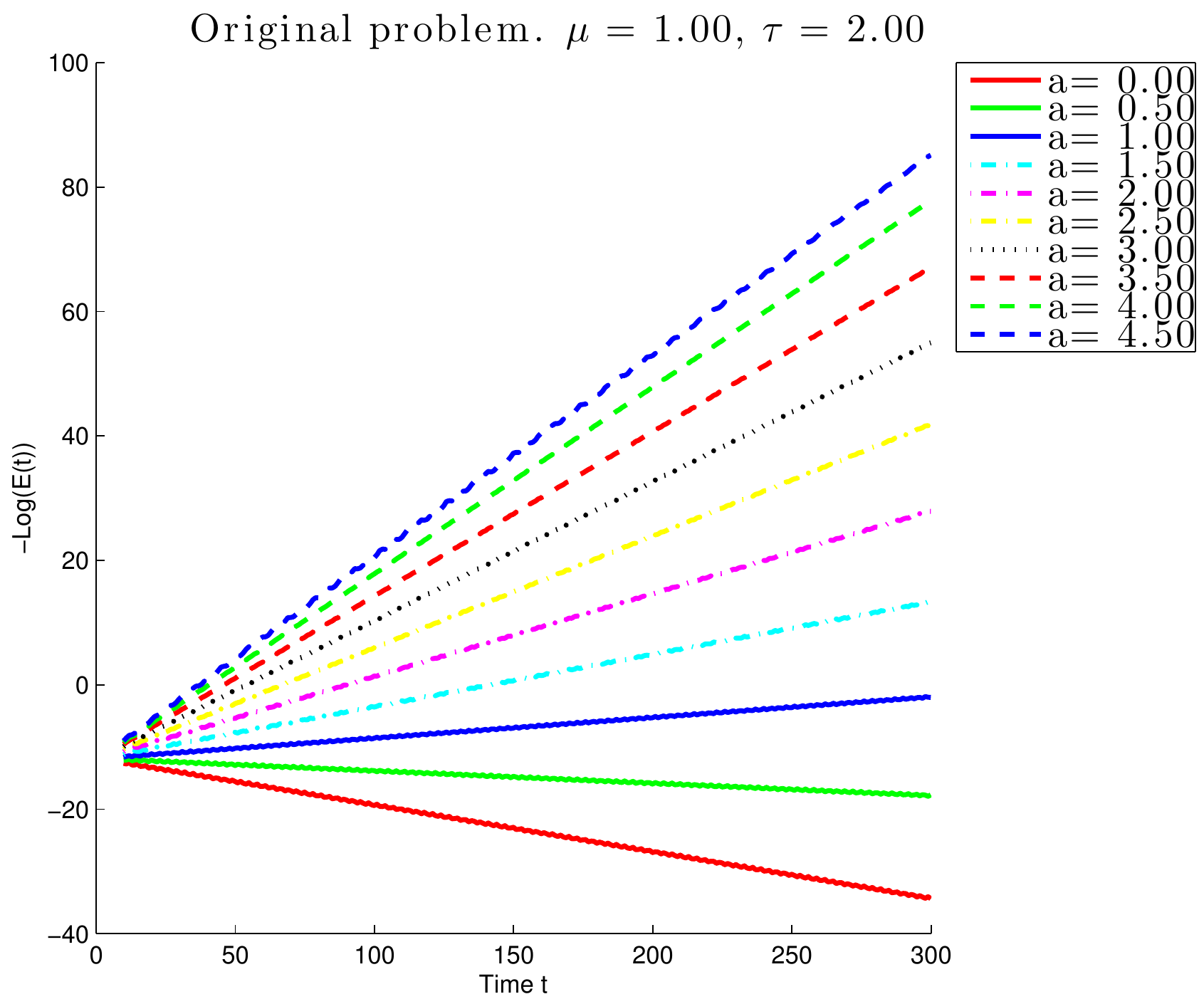}}
\subfigure[Shifted problem \label{shifteda}]{\includegraphics[scale=0.4]{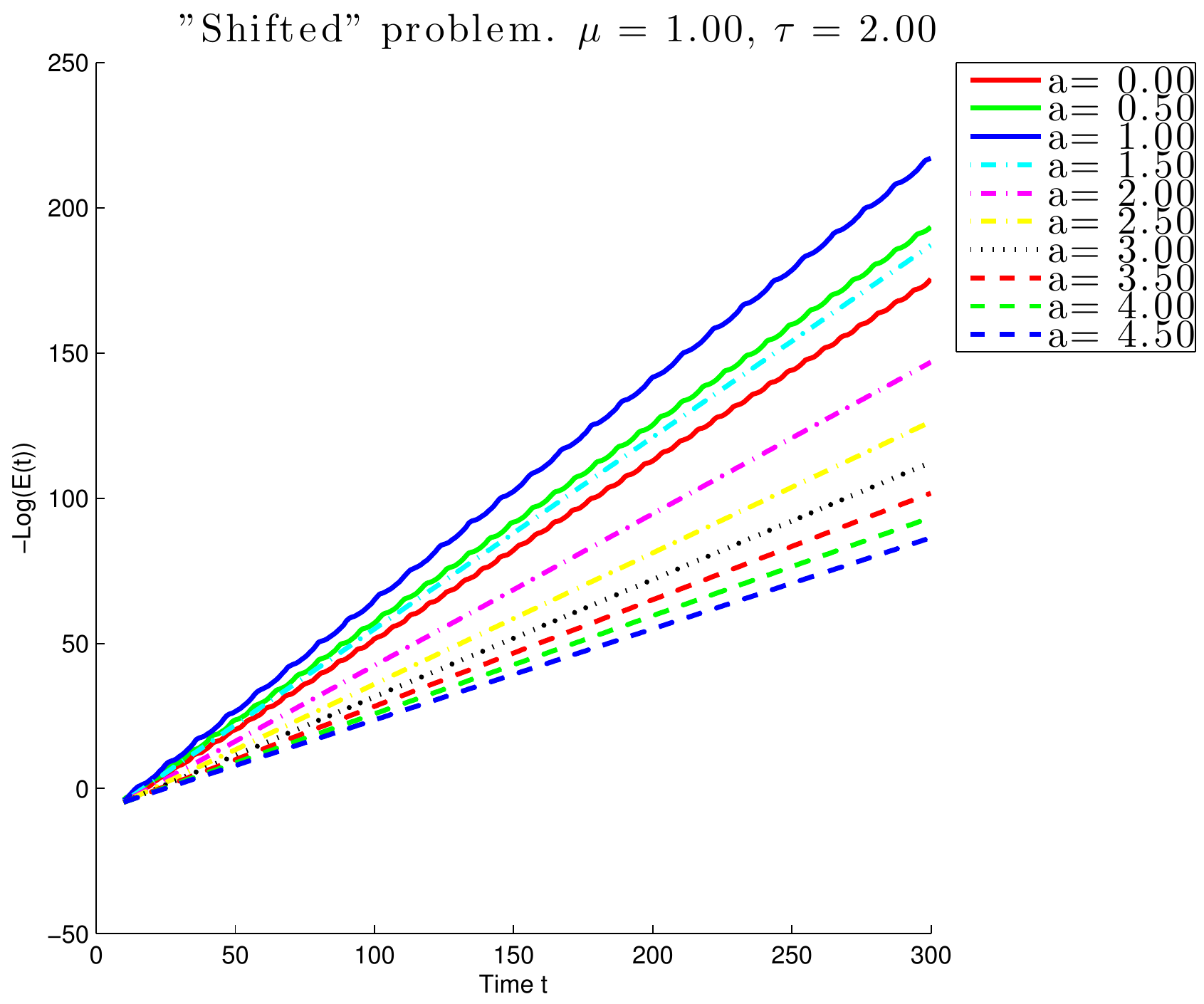}}
\caption{Energy (in -log scale) versus time: influence of $a$. \label{wave1a}}
\end{figure}
Let us first notice that the convergence rate for the shifted problem is better than the one expected: we have proved a polynomial decay rate
whereas numerically, we observe an exponential decay rate. 
This is probably due two facts :
\begin{enumerate}
\item the particular case of the dimension 1 as proved recently by 
G.Q. Xu, S. P. Yung and L. Kwan Li \cite{chinois-2006}.
\item the numerical diffusion participates to the exponential stability as for the Kelvin-Voigt damping.
\end{enumerate}

Moreover, as it was conjectured in Remark \ref{stabAd}, the ``shifted'' problem converges for a large set of parameters $a \mbox{ and }\mu$ whereas the original
problem does not and even worse, it exhibits an exponential growth.

In Figure \ref{kelvin}, we present the simulations for the case of the Kelvin-Voigt damping for which we have proved that under the condition
$ \mu < |c^{\star}| a $, we have an exponential decay rate. 

From equations \eqref{eigen-dirichlet-robin} and \eqref{defcstar},  the constant $c^{\star}$ must satisfy:
\begin{equation*}
\left\{\begin{array}{ll}
u_{xx} = 0\,,&  x \in (0,1), \\ 
u(0)=0\,,&   u_{x}(1) + c^{\star} u(1) = 0  \ .
\end{array}
\right.
\end{equation*}
Thus we obtain $c^{\star} = -1$. Let us first notice that even though the condition between $a$ and $\mu$ is not fulfilled,
we have an exponential decay of the solution. This is also probably due to the particular
case of the dimension 1 as well. Secondly it seems that numerically the convergence rate does not depend on the parameter $\mu$.
\begin{figure}[H]
\subfigure[Influence of $a$ \label{kelvina}]{\includegraphics[scale=0.4]{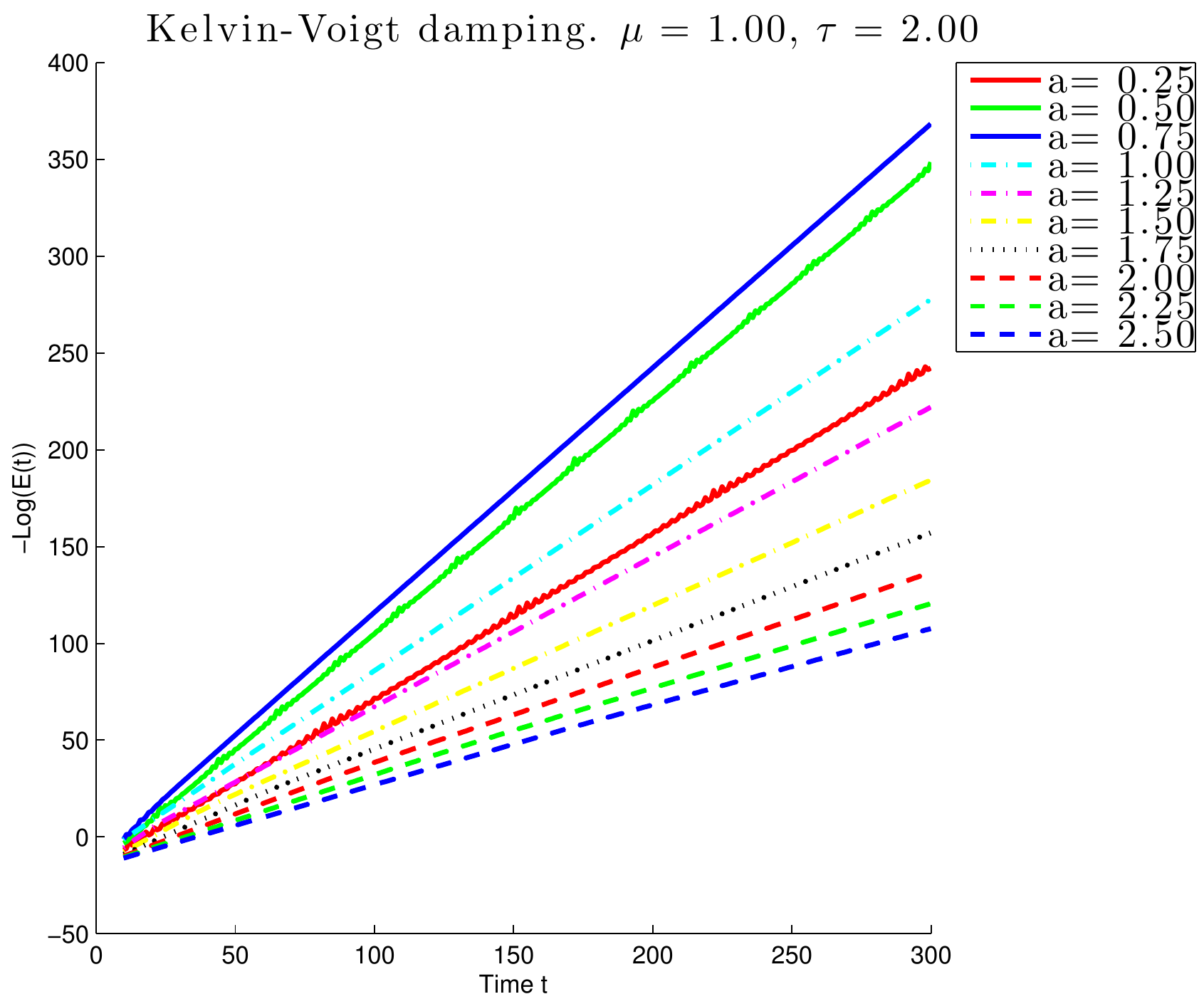}}
\subfigure[Influence of $\mu$ \label{kelvinmu}]{\includegraphics[scale=0.4]{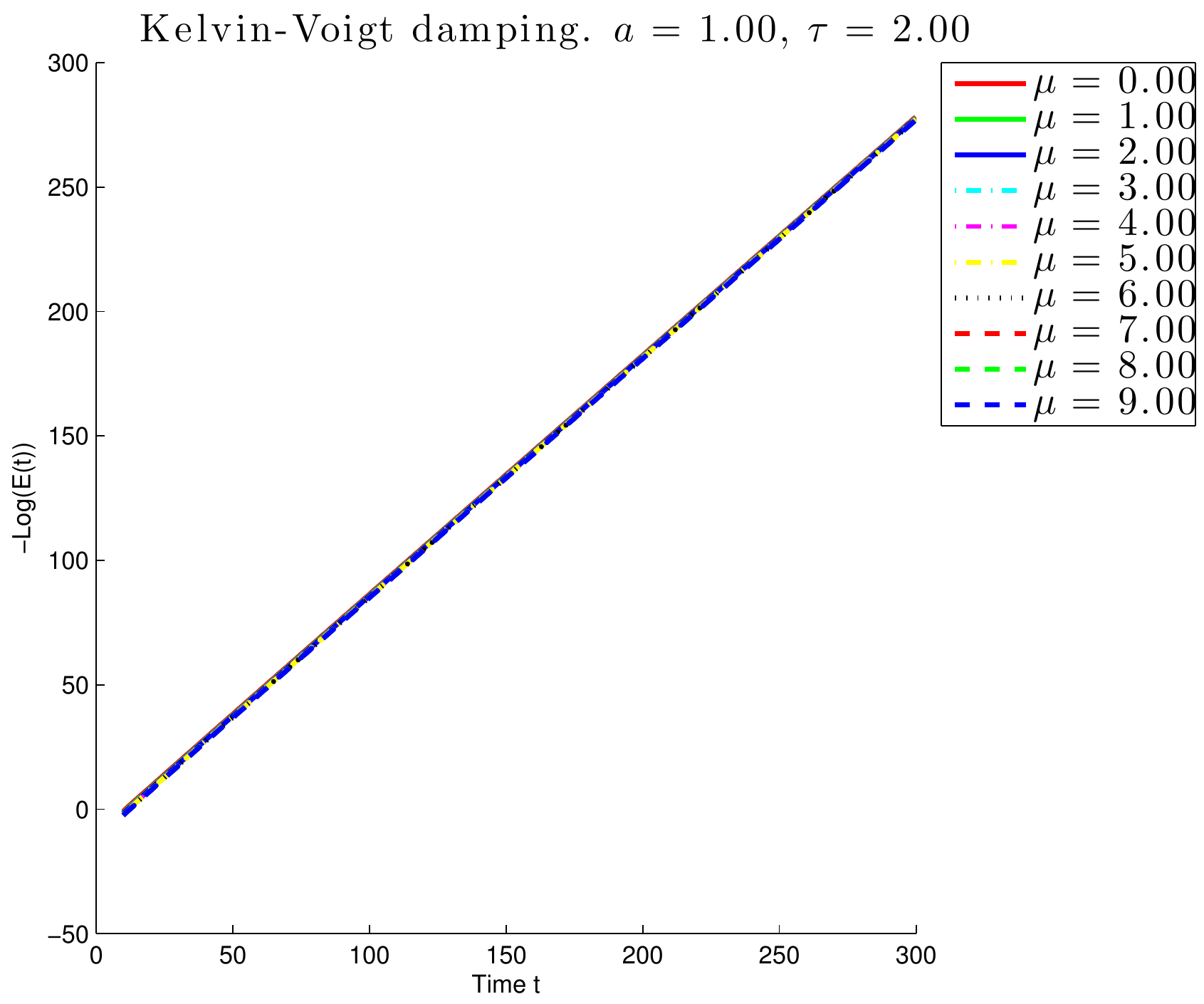}}
\caption{Energy (in -log scale) versus time. \label{kelvin}}
\end{figure}
\begin{ack} \rm 
The authors wish to thank R\'{e}gion Rh\^one-Alpes for the financial support COOPERA, CMIRA 2013.
\end{ack}
\bibliographystyle{siam}

\end{document}